\documentclass[a4paper,11pt]{amsart}

\newfont{\cyr}{wncyr10 scaled 1100}

\usepackage[left=2.7cm,right=2.7cm,top=3.5cm,bottom=3cm]{geometry}
\usepackage{amsthm,amssymb,amsmath,amsfonts,mathrsfs,amscd,yfonts}
\usepackage{turnstile}
\usepackage[latin1]{inputenc}
\usepackage[all]{xy}
\usepackage{latexsym}
\usepackage{longtable}
\usepackage[usenames,dvipsnames]{color}
\usepackage{hyperref}

\theoremstyle{plain}
\newtheorem{theorem}{Theorem}[section]

\newtheorem{proposition}[theorem]{Proposition}
\newtheorem{propo}[theorem]{Proposition}
\newtheorem{coro}[theorem]{Corollary}

\theoremstyle{definition}

\newtheorem{defi}[theorem]{Definition}
\newtheorem{question}[theorem]{Question}
\newtheorem{examplewr}[theorem]{Example}

\theoremstyle{remark}
\newtheorem{obswr}[theorem]{Observation}
\newtheorem{remarkwr}[theorem]{Remark}

\newenvironment{remark}{\begin{remarkwr}\begin{upshape}}{\end{upshape}\end{remarkwr}}

\DeclareMathOperator{\dR}{\mathrm{dR}}

\DeclareMathOperator{\BK}{BK}

\DeclareMathOperator{\et}{et}
\DeclareMathOperator{\id}{Id}

\DeclareMathOperator{\pr}{pr}
\DeclareMathOperator{\cyc}{cyc}
\DeclareMathOperator{\fin}{f}

\DeclareMathOperator{\nr}{nr}

\DeclareMathOperator{\Spf}{Spf}
\DeclareMathOperator{\alg}{alg}
\DeclareMathOperator{\grp}{grp}

\DeclareMathOperator{\ur}{ur}

\DeclareMathOperator{\Eis}{Eis}

\DeclareMathOperator{\ad}{ad}
\DeclareMathOperator{\Dd}{DD}
\DeclareMathOperator{\Katz}{Katz}

\newcommand{\cC}{\mathcal C}

\newcommand{\cU}{\mathcal U}
\newcommand{\cW}{\mathcal W}

\newcommand{\G}{\Gamma}

\newcommand{\Q}{\mathbb{Q}}
\newcommand{\Z}{\mathbb{Z}}

\newcommand{\Gal}{\mathrm{Gal\,}}
\newcommand{\GL}{\mathrm{GL}}

\newcommand{\Tr}{\mathrm{Tr}}

\newcommand{\Fil}{\mathrm{Fil}}

\newcommand{\Frob}{\mathrm{Fr}}
\newcommand{\End}{\mathrm{End}}

\newcommand{\Fr}{\mathrm{Fr}}

\newcommand{\ord}{{\mathrm{ord}}}
\newfont{\gotip}{eufb10 at 12pt}

\newcommand{\cO}{{\mathcal O}}

\newcommand{\cL}{{\mathcal L}}
\newcommand{\cS}{{\mathcal S}}

\newcommand{\cD}{{\mathcal D}}

\newcommand{\ra}{\rightarrow}
\newcommand{\lra}{\longrightarrow}

\newcommand{\hg}{{\mathbf{g}}}
\newcommand{\hh}{{\mathbf h}}

\newcommand{\cl}{{\mathrm{cl}}}

\DeclareMathOperator{\Hom}{Hom}

\newcommand{\res}{\mathrm{res}}

\newcommand{\mat}[4]{\left(\begin{array}{cc}#1&#2\\#3&#4\end{array}\right)}

\include{thebibliography}

\begin{document}

\title[Derived Beilinson--Flach elements and the arithmetic of the adjoint]{Derived Beilinson--Flach elements and the arithmetic of the adjoint of a modular form}

\author{\'Oscar Rivero and Victor Rotger}

\begin{abstract}
Kings, Lei, Loeffler and Zerbes constructed in \cite{LLZ}, \cite{KLZ} a three-variable Euler system $\kappa({\bf g},{\bf h})$ of Beilinson--Flach elements associated to a pair of Hida families $({\bf g},{\bf h})$ and exploited it to obtain applications to the arithmetic of elliptic curves, extending the earlier work \cite{BDR2}. The aim of this article is showing that this Euler system also encodes arithmetic information concerning the group of units of the associated number fields. The setting becomes specially novel and intriguing  when ${\bf g}$ and ${\bf h}$ specialize in weight $1$ to $p$-stabilizations of eigenforms such that one is dual of another. We encounter an exceptional zero phenomenon which forces the specialization of $\kappa({\bf g}, {\bf h})$ to vanish and we are led to study the {\em derivative} of this class. The main result we obtain is the proof of the main conjecture of \cite{DLR2} on iterated integrals and the main conjecture of \cite{DR2.5} for Beilinson--Flach elements  in the adjoint setting. The main point of this paper is that the methods of \cite{DLR1}, \cite{DLR2} and \cite{CH}, where the above conjectures are proved when the weight $1$ eigenforms have CM, do not apply to our setting and new ideas are required. In loc.\,cit.\,a crucial ingredient is a factorization of $p$-adic $L$-functions, which in our scenario is not available due to the lack of critical points. Instead we resort to the principle of {\em improved}  Euler systems and $p$-adic $L$-functions to reduce our problems to questions which can be resolved using Galois deformation theory. We expect this approach may be adapted to prove other cases of the Elliptic Stark Conjecture and of its generalizations that are appearing in the literature.
\end{abstract}

\address{O. R.: Departament de Matem\`{a}tiques, Universitat Polit\`{e}cnica de Catalunya, C. Jordi Girona 1-3, 08034 Barcelona, Spain}
\email{oscar.rivero@upc.edu}
\address{V. R. (corresponding author): Departament de Matem\`{a}tiques, Universitat Polit\`{e}cnica de Catalunya, C. Jordi Girona 1-3, 08034 Barcelona, Spain}
\email{victor.rotger@upc.edu}


\maketitle

\tableofcontents

\section{Introduction}

The purpose of this article is proving two results conjectured by H. Darmon, A. Lauder and one of the authors in \cite{DLR2} and \cite{DR2.5} respectively, concerning the interplay between the arithmetic of units in number fields, the theory of Coleman iterated integrals, and the Rankin $p$-adic $L$-function and Euler system of Beilinson--Flach elements associated to a pair $(\hg,\hh)$ of Hida families of modular forms.

The first theorem of this article is Theorem A (together with its equivalent form given in Theorem A'), which yields a proof of the main conjecture of \cite{DLR2} in the adjoint setting. Although Beilinson--Flach elements are certainly behind the scenes, we find interesting to remark that this Euler system is not involved neither in the statement nor the proof we provide of Theorems A, A', and can be phrased  in purely analytic terms by means of $p$-adic iterated integrals and $p$-adic $L$-functions. It is only later in the note that we explore the consequences that Theorem A has on the weight one specialisations of the Euler system of Beilinson-Flach elements, as described in Theorems B and C.

Let $\chi$ be a Dirichlet character of level $N\geq 1$ and let $M_k(N,\chi)$ (resp.\,$S_k(N,\chi)$) denote the space of (resp.\,cuspidal) modular forms of weight $k$, level $N$ and nebentype $\chi$. Let $g = \sum_{n\geq 1} a_n q^n  \in S_1(N,\chi)$ be a normalized newform and let $g^* = g\otimes \chi^{-1}$ denote its twist by the inverse of its nebentype. Let
\[ \varrho_g: \Gal(H_g/\Q) \hookrightarrow \GL(V_g) \simeq \GL_2(L), \quad \varrho_{\mathrm{ad}^0(g)}: \Gal(H/\Q) \hookrightarrow \GL(\mathrm{ad}^0(g)) \simeq \GL_3(L) \]
denote the Artin representations associated to $g$ and its adjoint, respectively. Here $H_g \supseteq H$ denote the finite Galois extensions of $\Q$ cut out by these representations, and $L$ is a sufficiently large finite extension of $\Q$ containing their traces. The three-dimensional representation $\mathrm{ad}^0(g)$ may be identified with the subspace of $\End(V_g)$ of null-trace endomorphisms, on which $G_\Q$ acts by conjugation. There is a natural decomposition of $L[G_\Q]$-modules $\ad(g) := \End(V_g) = L\oplus \mathrm{ad}^0(g)$, where $L$ stands for the trivial representation.

Fix a prime $p\nmid N$ and let $S_k^{\mathrm{oc}}(N,\chi)$ denote the space of overconvergent $p$-adic modular forms of weight $k$, tame level $N$ and character $\chi$. Fix an embedding $\bar\Q \subset \bar\Q_p$ and let $L_p$ denote the completion of $L$ in $\bar\Q_p$.

Label and order the roots of the $p$-th Hecke polynomial of $g$ as $X^2-a_p(g)X+\chi(p) = (X-\alpha)(X-\beta)$. We assume throughout that
\begin{enumerate}
\item[(H1)] The reduction of $\varrho_g$ mod $p$ is irreducible;

\item[(H2)] $g$ is $p$-distinguished, i.e.\,$\alpha \ne \beta \, (\mathrm{mod} \, p)$, and

\item[(H3)] $\varrho_g$ is not induced from a character of a real quadratic field in which $p$ splits.

\end{enumerate}

Let \[ g_\alpha(q) = g(q) -\beta g(q^p) \] denote the $p$-stabilization of $g$ on which the Hecke operator $U_p$ acts with eigenvalue $\alpha$. Enlarge $L$ if necessary so that it contains all Fourier coefficients of $g_\alpha$.

As shown in \cite{BeDi} and \cite{DLR1}, the above hypotheses ensure that any generalized overconvergent modular form with the same generalized eigenvalues as $g_\alpha$ is classical, and hence simply a multiple of $g_\alpha$. This allows to define a canonical projector
$$
e_{g_\alpha}: S_1^{\mathrm{oc}}(N,\chi) \lra L_p
$$
which extracts from an overconvergent modular form its coefficient at $g_\alpha$ with respect to an orthonormal basis: cf.\,\cite{DLR1} for more details.

One can attach a $p$-adic invariant $I_p(g)\in L_p$ to $g_\alpha$ as follows. Let
\[ Y:=Y_1(N) \quad \subset \quad X:= X_1(N) \]
denote the models over $\Q$ of the (affine and projective, respectively) modular curves classifying pairs $(A,P)$ where $A$ is a (generalised) elliptic curve and $P$ is a point of order $N$ on $A$.


For $a\in (\Z/N\Z)^\times$, let $\mathfrak{g}_{a;N}$ be Kato's {\em Siegel unit} whose $q$-expansion is given by
\begin{equation}\label{def-gaM}
\mathfrak{g}_{a;N} := (1-\zeta_N^a) q^{1/12} \prod_{n=1}^\infty (1-\zeta_N^a q^n) (1-\zeta_N^{-a} q^n), \end{equation}
where $\zeta_N$ is a fixed primitive $N$-th root of unity.
The unit $\mathfrak{g}_{a;N}$
can naturally be viewed as belonging to
$\Q  \otimes \cO_{Y}^\times$, and its $q$-expansion is defined over $\Q(\mu_N)$.

Set
\begin{equation}\label{eqn:def-g-chi}
\mathfrak{g} = \frac{1}{2}\otimes \prod_{a=1}^N \mathfrak{g}_{a;N} \in \Q  \otimes \cO_{Y}^\times,
\end{equation}
and define
\[ E_0 = \log_p \mathfrak{g}. \]
This is a locally analytic modular form of weight $0$ on $Y$. It follows from \eqref{def-gaM} and \eqref{eqn:def-g-chi} that the logarithmic derivative of $\mathfrak{g}$ is
\begin{equation}
\label{dlog=h}
d E_0 = E_2 \frac{dq}{q} \in \Omega^1_{Y},
\end{equation}
where
\begin{equation*}
E_2 = \frac{1}{2} \cdot \zeta(-1) + \sum_{n=1}^\infty  \left(\sum_{d|n} d \right) q^n
\end{equation*}
is the classical Eisenstein series of weight $2$.

Let $E_0^{[p]} := \sum_{p\nmid n} a_n(E_0) q^n$ denote the $p$-depletion of $E_0$; this is an overconvergent modular form and we may define
\begin{equation}\label{L-invariant}
I_p(g) := e_{g_\alpha}(E_0^{[p]}\cdot g_\alpha).
\end{equation}

The invariant $I_p(g)$ can be recast as a $p$-adic iterated integral, denoted $I'_p(E_2,g,g^*)$ in \cite{DLR2}. The first main result of this note, conjectured in \cite{DLR2}, asserts that $I_p(g)$ is equal to the following motivic expression. As shown in \cite[Lemma 1.1]{DLR2}, we have
\[
\dim_L (\cO_H^\times \otimes \mathrm{ad}^0(g))^{G_\Q} = 1, \quad \dim_L (\cO_H[1/p]^\times/p^{\mathbb Z} \otimes \mathrm{ad}^0(g))^{G_\Q} = 2.
\]
Fix a generator $u$ of $ (\cO_H^\times \otimes \mathrm{ad}^0(g))^{G_\Q}$ and also an element $v$ of $(\cO_H^{\times}[1/p]^{\times} \otimes \mathrm{ad}^0(g))^{G_\Q}$ in such a way that $\{ u, v\}$ is a basis of $(\cO_H[1/p]^\times/p^\Z \otimes \mathrm{ad}^0(g))^{G_\Q}$. The element $v$ may be chosen to have $p$-adic valuation $\ord_p (v)=1$, and we do so.

Viewed as a $G_{\Q_p}$-module, $\mathrm{ad}^0(g)$ decomposes as $\mathrm{ad}^0(g) = L \oplus L^{\alpha \otimes \bar \beta} \oplus L^{\beta \otimes \bar \alpha}$, where each line is characterized by the property that the arithmetic Frobenius $\Frob_p$ acts on it with eigenvalue $1$, $\alpha/\beta$ and $\beta/\alpha$, respectively\footnote{The decomposition of $\mathrm{ad}^0(g)$ as the direct sum of three canonical lines is also available when $\alpha/\beta=\beta/\alpha=-1$, see \eqref{gh-decom} and \eqref{decom-uv} for details.}. Let $H_p$ denote the completion of $H$ in $\bar\Q_p$  and let
\[
u_1, \,\, u_{\alpha \otimes \bar \beta},  \,\,  u_{\beta \otimes \bar \alpha},  \,\,  v_1,  \,\,  v_{\alpha \otimes \bar \beta},  \,\,  v_{\beta \otimes \bar \alpha} \in H_p^\times \otimes_{\Q} L \quad (\mathrm{mod} \, L^\times)
\]
denote the projection of the elements $u$ and $v$ in $(H_p^\times \otimes \mathrm{ad}^0(g))^{G_{\Q_p}}$ to the above lines: cf.\,\eqref{decom-uv} for more details. By construction we have $u_1, v_1 \in \Q_p^\times$ and
\[
\Frob_p(u_{\alpha \otimes \bar \beta}) = \frac{\beta}{\alpha} u_{\alpha \otimes \bar \beta}, \, \Frob_p(v_{\alpha \otimes \bar \beta}) = \frac{\beta}{\alpha} v_{\alpha \otimes \bar \beta}, \, \Frob_p(u_{\beta \otimes \bar \alpha}) = \frac{\alpha}{\beta} u_{\beta \otimes \bar \alpha}, \, \Frob_p(v_{\beta \otimes \bar \alpha}) = \frac{\alpha}{\beta} v_{\beta \otimes \bar \alpha}.
 \]

Let
\[
\log_p: H_p^\times \otimes L \lra H_p \otimes L
\]
denote the usual $p$-adic logarithm.

\vspace{0.2cm}

{\bf Theorem A.} {\em
The following equality holds in $L_p \subset H_p \otimes L \quad (\mathrm{mod} \, L^\times)$:
\begin{equation}\label{main1}
I_p(g) =  \frac{1}{\log_p(u_{\alpha \otimes \bar \beta})} \times \det \begin{pmatrix} \log_p u_1 & \log_p u_{\alpha \otimes \bar \beta} \\ \log_p v_1 & \log_p v_{\alpha \otimes \bar \beta} \end{pmatrix}.
\end{equation}
}

\vspace{0.2cm}

This is \cite[Conjecture 1.2]{DLR2} specialized to the pair $(g,g^*)$: note that  both statements are equivalent because $(\cO_H^\times \otimes V_g^\alpha)^{G_{\Q_p}} \simeq \Hom_{G_{\Q_p}}(V_g^\beta,\cO_H^\times)$.

Let $\mathcal L(\ad^0(g_\alpha))$ be Hida's analytic $\mathcal L$-invariant attached to the adjoint Galois representation of $g_\alpha$; cf.\,\eqref{citro}. As we recall in Proposition \ref{impad}, there are several equivalent definitions of this invariant.


These $\cL$-invariants are in general difficult to compute, and explicit expressions for them are rather rare. The proof of the above result, which we describe further below in this introduction, shows that Theorem A may be recast as a formula for $\mathcal L(\ad^0(g_\alpha))$. Since this might be of independent interest, specially for the reader less acquainted with the Stark elliptic conjectures of \cite{DLR1}, \cite{DLR2} and more familiar with the theory of $\cL$-invariants, we quote it below as a separate statement:

\vspace{0.2cm}

{\bf Theorem A'.} {\em
The analytic $\mathcal L$-invariant of $\ad^0(g_\alpha)$ is
\begin{equation}\label{main1'}
\mathcal L(\ad^0(g_{\alpha})) =  \frac{ \log_p(u_1)\log_p(v_{\alpha \otimes \bar \beta}) - \log_p(u_{\alpha \otimes \bar \beta})\log_p(v_1)}{\log_p(u_{\alpha \otimes \bar \beta})} \quad (\mathrm{mod} \, L^\times).
\end{equation}
}

\vspace{0.2cm}

\begin{remark} The prototypical case where an explicit formula for the $\cL$-invariant is known arises of course when $E/\Q$ is an elliptic curve of split multiplicative reduction at $p$. In that case
$\mathcal L_p(E) = \frac{\log_p(q_E)}{\ord_p(q_E)}$ where $q_E \in p \mathbb Z_p$ is a Tate period for $E$. Observe that this expression and the one in Theorem A' are in fact very similar, as both may be recast as
\begin{equation}\label{Lgen}
\cL = \frac{\log_p(\kappa)}{\ord_p(\kappa)}
\end{equation}
where $\kappa$ is an element in $H^1(\Q_p,\Q_p(1))$ arising from some global class in motivic cohomology. In the classical case one has $\kappa = q_E$, which might be regarded as an element in the {\em extended} Mordell-Weil group $\tilde E(\Q)$, or Nekov\'ar's extended Selmer group of $E$. In Theorem A' one has, up to scalar, $\kappa=\log_p(v_{\alpha \otimes \bar \beta})u_1 - \log_p(u_{\alpha \otimes \bar \beta}) v_1$ in the unit group of $H$; note that recipe \eqref{Lgen} indeed gives rise to \eqref{main1'} because $\ord_p(u_1)=0$ and $\ord_p(v_1)=1$.
\end{remark}

The second main result of this note may be regarded as a conceptual explanation of Theorem A, as it establishes a connection between the iterated integral $I_p(g)$, a special value of Hida's $p$-adic Rankin $L$-function associated to the pair $(g,g^*)$, and the generalized Kato classes arising from the three-variable Euler system of Beilinson--Flach elements constructed in \cite{KLZ}.

In order to describe it more precisely, let $\Lambda = \Z_p[[\Z^\times_p]]$ denote the Iwasawa algebra and denote  $\mathcal W = \Spf(\Lambda)$ the associated weight space. As shown by Wiles in \cite{Wi} (cf.\,also \cite[\S 7.2]{KLZ} for the normalizations we adopt), associated  to $g_\alpha$ there are:
\begin{enumerate}
\item a  finite flat extension $\Lambda_{\hg}$ of $\Lambda$, giving rise to a covering $\mathrm{w}: \mathcal W_{{\bf g}} = \Spf(\Lambda_{\hg}) \lra \cW$;

\item a family of overconvergent $p$-adic ordinary modular forms ${\bf g}$ with coefficients in $\Lambda_{\hg}$ specializing to $g_{\alpha}$ at some point $y_0\in \mathcal W_{{\bf g}}$ of weight $\mathrm{w}(y_0) = 1$; our running assumptions imply that the above Hida family  passing through $g_\alpha$ is unique by \cite{BeDi}.

\item a Galois representation $\varrho_{\hg}: G_\Q \lra  \GL_2(\Lambda_{\hg})$ characterized by the property that all its classical specializations coincide with Deligne's Galois representation associated to the corresponding specialization of the Hida family.
\end{enumerate}

Let $\mathbb V_{\hg}$ denote the rank two $\Lambda_{\hg}$-module realizing the Galois representation $\varrho_{\hg}$ as specified e.g.\,in \cite[\S 7.2]{KLZ}, where this is denoted $M(\hg)^*$. Let ${\bf g}^* := \hg \otimes \chi^{-1}$ denote the twist of $\hg$ by the inverse of its tame nebentype. Note that ${\bf g}^*$ specializes at $y_0$ to the eigenform $g_{\alpha} \otimes \chi^{-1} = g_{1/\beta}^*$, namely the $p$-stabilization of $g^*$ on which $U_p$ acts with eigenvalue $1/\beta$.

Let \[ \varepsilon_{\cyc}: G_{\Q} \rightarrow \Gal(\Q(\mu_{p^{\infty}})/\Q) \rightarrow  \Z_p^\times \] denote the $p$-adic cyclotomic character. Let $\underline{\varepsilon}_{\cyc}$ be the composition of $ \varepsilon_{\cyc}$ with the natural inclusion $\Z_p^\times \subset \Lambda^{\times}$ taking $z$ to the group-like element $[z]$ in $ \Lambda^{\times}$.

Define the three-variable Iwasawa algebra  $\Lambda_{{\bf g g^*}}:=\Lambda_{{\bf g}} \hat \otimes_{\mathbb Z_p} \Lambda_{{\bf g^*}} \hat \otimes_{\mathbb Z_p} \Lambda$  and  the $\Lambda_{\hg \hg^*}[G_\Q]$-module $$\mathbb V_{{\hg \hg^*}}:=\mathbb V_{{\hg}} \hat\otimes_{\mathbb Z_p} \mathbb V_{{\hg^*}} \hat\otimes_{\mathbb Z_p} \Lambda(\varepsilon_{\cyc} \underline{\varepsilon}_{\cyc}^{-1}).$$

The article \cite{KLZ} attaches to $({\bf g},{\bf g}^*)$ a $\Lambda$-adic global cohomology class \[ \kappa({\bf g},{\bf g}^*) \in H^1(\Q,\mathbb V_{{\hg \hg^*}}) \] parametrized by the triple product of weight spaces $\cW_{\hg \hg^*} := \mathcal W_{{\bf g}} \times \mathcal W_{{\bf g}} \times \mathcal W$, where the first two variables are afforded by the weight of the Hida families ${\bf g}$ and ${\bf g}^*$, and the third one is the cyclotomic variable.

The common key strategy in several recent works on the arithmetic of elliptic curves (\cite{BDR2}, \cite{DR2},\cite{LLZ},\cite{KLZ},\cite{CH}) consists in proving a {\em reciprocity law} relating a suitable specialization of a $\Lambda$-adic class as the one above to the critical value of the underlying classical $L$-function. The non-vanishing of the appropriate (classical or $p$-adic) $L$-value can then be invoked to show that the associated global cohomology class is not trivial, and one can derive striking arithmetic consequences from this (cf.\,e.g.\,\cite[Theorems A and B]{DR2}, \cite[Corollary C]{KLZ}).

Our setting differs in several key aspects from the previous ones. Namely, while it is also natural to consider the specialization
\begin{equation}\label{kappa}
\kappa(g_{\alpha},g_{1/\beta}^*):=\kappa({\bf g},{\bf g}^*)(y_0,y_0,0) \in H^1(\Q,V_g\otimes V_{g^*}(1))
\end{equation}
of the Euler system of Beilinson--Flach elements
at the point $(y_0,y_0,0)$, one can show that this global cohomology class is trivial
and thus no arithmetic information can be extracted directly from it: cf.\,Theorem \ref{classes-zero}.

The fact that $\kappa(g_{\alpha},g_{1/\beta}^*)=0$ might be regarded as an {\em exceptional zero phenomenon}, albeit quite different from the ones one typically encounters in e.g.\,the pioneering work \cite{MTT} of Mazur, Tate and Teitelbaum, because our case corresponds to a point lying {\em outside} the classical region of interpolation of the associated Hida--Rankin $p$-adic $L$-function. But it still keeps the same flavor, because the vanishing of the global cohomology class \eqref{kappa} is caused by the cancellation of the Euler-like factor at $p$ arising in the interpolation process of the construction of the Euler system. To be more precise, the situation is more delicate than that: for classical points $y\in \cW_{\hg}$ of weights $\ell>1$, the triples $(y,y,\ell-1)$ do lie in the region of {\em geometric} interpolation of the Euler system of Beilinson--Flach elements, and the main theorem of \cite{KLZ} applies and implies that $\kappa({\bf g},{\bf g}^*)(y,y,\ell-1) = 0$. The vanishing of $\kappa(g_{\alpha},g_{1/\beta}^*)$ then follows by a density argument exploiting various $\Lambda$-adic Perrin-Riou regulators; we refer to \S \ref{self-dual} for more details.

We are hence placed to work with the {\em derived class} $\kappa'(g_{\alpha},g_{1/\beta}^*) \in  H^1(\mathbb Q, \mathrm{ad}^0(g)(1))$, which is introduced in \S \ref{deri} as the derivative along the weight of $\hg^*$ of the $\Lambda$-adic class $\kappa({\bf g},{\bf g}^*)$ at $(y_0,y_0,0)$.

In \cite{DR2.5} it was laid a conjecture proposing an explicit description of the generalized Kato classes arising from the Euler systems of diagonal cycles of \cite{DR2}. This conjecture, together with the expected (but so far also unproved) behavior of certain periods arising from Hida theory, is shown to imply the main conjecture of \cite{DLR1} for twists of elliptic curves by Artin representations. This is seen to provide a conceptual interpretation of the numerical examples computed in  \cite{DLR1}.

At the time of writing  \cite{DR2.5}, the authors had in mind that a parallel formulation should also hold for the Euler system of Beilinson--Flach elements attached to pairs $(g,h)$ of eigenforms of weights $(1,1)$ in connection with \cite{DLR2}. This is indeed the case for arbitrary pairs $(g,h)$ such that $h\ne g^*$, as there are no exceptional zero phenomena and the constructions and conjectures \cite{DR2.5} can be adapted to the setting of units of number fields. We refer to \cite{RR} for the details.

However, pairs $(g,g^*)$ present a completely different scenario, as we already hinted at above. Not only $\kappa(g_{\alpha},g_{1/\beta}^*) = 0$, but it also turns out that the derived class $\kappa'(g_{\alpha},g_{1/\beta}^*)$ is {\em not} crystalline at $p$, as opposed to the set-up of \cite{DR2.5} and \cite{RR}. Yet it is still possible to formulate a conjecture analogous to loc.\,cit.,\,proposing an explicit description of $\kappa'(g_{\alpha},g_{1/\beta}^*)$ in terms of the $\mathrm{ad}^0(g)$-isotypical component of the group of $p$-units of $H$. We can in fact prove it, giving rise to the second main result of this paper.

Let $H_{\fin,p}^1(\mathbb Q, \mathrm{ad}^0(g)(1))$ denote the subspace of $H^1(\mathbb Q,\mathrm{ad}^0(g)(1))$ consisting of classes that are de Rham at $p$ and unramified at all remaining places. Kummer theory (cf.\,Proposition \ref{Kummer} below) gives rise to a canonical isomorphism
\begin{equation}\label{kum}
 \mathcal O_H[1/p]^{\times}[\mathrm{ad}^0(g)] \otimes L_p \xrightarrow{\simeq} H_{\fin,p}^1(\mathbb Q, \mathrm{ad}^0(g)(1))
\end{equation}
and we shall use \eqref{kum} throughout to identify these two spaces.



\vspace{0.2cm}

{\bf Theorem B.} {\em
Assume that $\mathcal L(\ad^0(g_{\alpha})) \neq 0$. The equality
\begin{equation}\label{main2}
\kappa'(g_{\alpha},g_{1/\beta}^*) =  \mathcal L(\ad^0(g_{\alpha})) \cdot  \frac{ \log_p(v_{\alpha \otimes \bar \beta}) u - \log_p(u_{\alpha \otimes \bar \beta})v}{\log_p(u_{\alpha \otimes \bar \beta})}
\end{equation}
holds in $\frac{\mathcal O_H[1/p]^{\times}}{p^\Z}[\mathrm{ad}^0(g)] \otimes L_p  \quad (\mathrm{mod} \, L^\times)$.
}

\vspace{0.2cm}

Here it again becomes apparent the notable differences between the phenomena occurring in say \cite{BDR2}, \cite{DR2}, \cite{KLZ} and this article. In loc.\,cit.,\,the non-vanishing of the central critical $L$-value is shown to imply the triviality of the Mordell-Weil group; conversely, the vanishing of the central $L$-value should imply the non-triviality of the Mordell-Weil group.

In contrast, in our setting we have $L(g,g^*,1) \ne 0$ and nonetheless the associated motivic groups $\mathcal O_H^{\times}[\mathrm{ad}^0(g)]$ and $\mathcal O_H[1/p]^{\times}[\mathrm{ad}^0(g)]$ have positive rank. This is ultimately explained by the fact that $L(g,g^*,s)=L(\ad^0(g),s) \cdot \zeta(s)$ factors as the product of two zeta functions having a simple zero and a simple pole at $s=1$, respectively.

\vspace{0.2cm}
\noindent

\begin{remark}
As we show in \S \ref{proof-main}, it is actually possible to prove that $\mathcal L(\ad^0(g_{\alpha})) \neq 0$ in many dihedral cases, where $\varrho_g$ is induced from a finite order character of the Galois group of a (real or imaginary) quadratic field. We refer to loc.\,cit.\,for more details. On the other hand, when $g$ is exotic, we still expect the $\cL$-invariant not to vanish systematically, but proving this properly appears to be a less accessible question in the theory of transcendental $p$-adic numbers.
\end{remark}

The main ingredients in the proof of Theorems A, A' and B are the following:

\vspace{0.2cm}
\noindent

{\it {\bf (I)}}  Let $L_p({\bf g},{\bf g^*})$ be the three-variable Hida--Rankin $p$-adic $L$-function as introduced in \cite{Hi2} or \cite[\S 3.6]{Das3}, and let $L_p(g_{\alpha},g_{1/\beta}^*,s):=L_p({\bf g},{\bf g^*})(y_0,y_0,s)$ denote the restriction to $\{ y_0,y_0\} \times \cW$. This function satisfies a functional equation relating the values at $s$ and $1-s$.

As explained in \cite{DLR2}, the $p$-adic iterated integral introduced in \eqref{L-invariant} may be recast as a special value of a $p$-adic $L$-function, namely
\cite[Lemma 4.2]{DLR2} asserts that under our running assumptions H1-H2-H3 that are in place throughout the article, we have
\[ I_p(g)  = L_p(g_{\alpha},g_{1/\beta}^*,0) = L_p(g_{\alpha},g_{1/\beta}^*,1) \quad (\mbox{mod }\, L^\times). \]

Define $\cS_{\mathrm{Hida}} \subset \mathcal W_{{\bf g}} \times \mathcal W_{{\bf g}} \times \mathcal W$ as the surface for which the collection of crystalline points $(y,z,s)$ of weights $(\ell,m,m)$ is Zariski dense. The Euler-like factor at $p$ showing-up in the interpolation formula of Hida--Rankin's $p$-adic $L$-function turns out to be a rigid-analytic Iwasawa function when restricted to $\mathcal S_{\mathrm{Hida}}$. In $\S$2.3 we recall how this observation allows  to conclude  thanks to \cite{Das3} that
\begin{equation}\label{LLL}
\mathcal L(\ad^0(g_{\alpha})) := \alpha'_{\hg}(y_0) = I_p(g) \pmod{L^{\times}}.
\end{equation}

Computing the latter derivative is a question which can be reduced to a problem in Galois deformation theory as in \cite{DLR3}, although the methods in loc.\,cit.\,do not apply directly and need to be adapted in order to cover our setting. This is carried out in \S \ref{sec:deform}, and allows us to prove Theorems A and A', as we spell out in detail in \S \ref{proof-main}.

\vspace{0.2cm}
\noindent

{\it {\bf (II)}}
The main result of \S \ref{sec-derived} is Theorem \ref{teobf}, which we state here in slightly rough terms (as the Bloch--Kato logarithm map appearing below has not been specified; cf.\,loc.\,cit.\,for more details):

{\bf Theorem C.} (Theorem \ref{teobf}) {\em
The {\it derived} Beilinson--Flach element satisfies
\begin{equation}\label{1a}
\log_p(\kappa'(g_{\alpha},g_{1/\beta}^*)) = \mathcal L(\ad^0(g_{\alpha})) \cdot L_p(g_{\alpha},g_{1/\beta}^*,0)  \pmod{L^{\times}}.
\end{equation}
}
We may see this as an {\it exceptional zero formula}, reminiscent of the main theorem of \cite{GS}, which asserts that for an elliptic curve $E/\Q$ with split multiplicative reduction at $p$, we have
\[
L_p'(E,1) = \mathcal L_p(E) \cdot L_{\alg}(E,1).
\]

We thus may invoke Theorem A, which combined with  \eqref{1a} implies that
$$
\log_p(\kappa'(g_{\alpha},g_{1/\beta}^*)) = \Big( \, \frac{ \log_p(u_1)\log_p(v_{\alpha \otimes \bar \beta}) - \log_p(u_{\alpha \otimes \bar \beta})\log_p(v_1)}{\log_p(u_{\alpha \otimes \bar \beta})} \, \Big)^2  \pmod{L^{\times}}.
$$
In light of the general properties satisfied by Beilinson--Flach elements established in \cite{KLZ}, this allows us to prove in \S \ref{proof-main} that the following equality holds in $\mathcal O_H[1/p]^{\times}[\mathrm{ad}^0(g)] \otimes L_p$:
$$
\kappa'(g_{\alpha},g_{1/\beta}^*) =  \frac{ \log_p(u_1)\log_p(v_{\alpha \otimes \bar \beta}) - \log_p(u_{\alpha \otimes \bar \beta})\log_p(v_1)}{\log_p(u_{\alpha \otimes \bar \beta})^2} \times  \Big(  \log_p(v_{\alpha \otimes \bar \beta}) u - \log_p(u_{\alpha \otimes \bar \beta})v \Big)
$$
up to a factor in $L^\times$, as claimed.
\vspace{0.2cm}
\noindent

We close the paper in  \S \ref{sec:DD} by pointing out to the connection between our derived Beilinson--Flach classes, Darmon--Dasgupta units and the Artin $p$-adic $L$-functions associated to the motives in play. We prove a factorization theorem in the CM case, and explain how our main results shed some light in the much more intriguing RM setting.

As a concluding remark, we hope the approach introduced in this paper may be adapted to prove other instances of variants of the Elliptic Stark Conjecture in presence of exceptional zeroes. We also hope that the derived global cohomology classes appearing in this work may also be useful for the on-going research project \cite{DHRV} of Darmon, Harris, Rotger and Venkatesh on derived Hecke  algebras in weight one.

\vspace{0.2cm}
\noindent

{\bf Acknowledgements.} It is a pleasure to thank H. Darmon and A. Lauder for many conversations around the theme of this article along the years; although  \cite{DLR1} was written before than  \cite{DLR2}, it was actually the numerical computations pertaining to the latter that led the authors to the idea that the $p$-adic iterated integral considered in loc.\,cit.\,ought to be non-trivial in scenarios of rank two. More specifically, \cite[Example 6.1]{DLR2} was the first ever numerical verification of the conjectures in a ``genuine rank two" case (as opposed to scenarios where the order of vanishing is $2$ because the $L$-function factors as the product of two $L$-functions of rank $1$). Note that this example falls within our results, so this is now a theorem.

We  heartily thank D. Loeffler for his useful and encouraging comments and explanations on his joint papers with S. Zerbes  \cite{LZ2}, which was also one of the initial sources of motivation for the writing of this note, as it points out to the interesting phenomenon of exceptional vanishing of classical specializations of the Euler system of Beilinson--Flach elements.  We also thank D. Benois for his interest and comments on an earlier draft of this note. The first author also wishes to thank F. Castell\`a and C. Skinner for their hospitality at Princeton University, where part of this research was conducted. We would like to thank the anonymous referee for a very careful reading of the manuscript, whose comments notably contributed to improve the exposition and correct a few number of inaccuracies in an earlier version of the manuscript.

Both authors were supported by Grant MTM2015-63829-P. This project has received funding from the European Research Council (ERC) under the European Union's Horizon 2020 research and innovation programme (grant agreement No. 682152). The first author has also received financial support through ``la Caixa" Fellowship Grant for Doctoral Studies (grant LCF/BQ/ES17/11600010). The second author gratefully acknowledges Icrea for financial support through an Icrea Academia award.

\section{Preliminary concepts}

\subsection{Hida families}\label{Hidafam}
Fix an algebraic closure $\bar{\mathbb Q}$ of $\mathbb Q$. For a number field $K$, let $G_K:=\Gal(\bar{\mathbb Q}/K)$ denote its absolute Galois group. Fix a prime $p$ and an embedding $\bar{\mathbb Q} \hookrightarrow \bar{\mathbb Q}_p$, and let $\ord_p$ denote the resulting $p$-adic valuation on $\bar{\mathbb Q}^{\times}$, normalized in such a way that $\ord_p(p)=1$.

Let $g$ be a newform of weight $k \geq 1$, level $N$ and character $\chi$, with Fourier coefficients in a finite extension $L$ of $\mathbb Q$. Label the roots of the $p$-th Hecke polynomial of $g$ as $\alpha_{g}, \beta_{g}$  with $\ord_p(\alpha_{g}) \leq \ord_p(\beta_g)$. By enlarging $L$ if necessary, we shall assume throughout that $L$ contains $\alpha_g, \beta_g$, the $N$-th roots of unity and the pseudo-eigenvalue $\lambda_N(g)$ with respect to the Atkin-Lehner operator $W_N$ (cf.\,\cite{AL}).

Let $L_p$ denote the completion of $L$ in $\bar\Q_p$ and let $V_g$ denote the two-dimensional representation of $G_\Q$ with coefficients in $L_p$ associated to $g$ as defined e.g.\,in \cite[\S 2.8]{KLZ}, where it is denoted $M_{L_p}(g)^*$. If $g$ is ordinary at $p$, there is an exact sequence of $G_{\mathbb Q_p}$-modules
\begin{equation}\label{filtration-classical}
 0 \rightarrow V_g^+ \rightarrow V_g \rightarrow V_g^- \rightarrow 0, \qquad V_g^+ \simeq L_p(\varepsilon_{\cyc}^{k-1}\chi \psi_g^{-1}), \quad V_g^- \simeq L_p(\psi_g),
\end{equation}
where $\psi_g$ is the unramified Galois character of $G_{\mathbb Q_p}$ sending $\Fr_p$ to $\alpha_g$.

The formal spectrum $\mathcal W = \Spf(\Lambda)$ of the Iwasawa algebra $\Lambda=\mathbb Z_p[[\mathbb Z^\times_p]]$ is called the weight space attached to $\Lambda$, and their $A$-valued points over a $p$-adic ring $A$ are given by \[ \mathcal W(A)=\Hom_{\alg}(\Lambda,A)=\Hom_{\grp}(\mathbb Z^\times_p, A^{\times}). \]

Weight space is equipped with a distinguished class of {\it arithmetic points} $\nu_{s,\varepsilon}$ indexed by integers $s \in \Z$ and Dirichlet characters $\varepsilon:(\Z/p^r\Z)^\times \rightarrow  \bar \Q^{\times}$ of $p$-power conductor. The point $\nu_{s,\varepsilon} \in \mathcal W$ is defined by the rule \[ \nu_{s,\varepsilon}(n)=\varepsilon(n)n^s. \]
We let $\cW^{\cl}$ denote the subset of $\cW$ formed by such arithmetic points.
When $\varepsilon=1$ is the trivial character, we denote the point $\nu_{s,1}$ simply as $\nu_s$ or even $s$ by a slight abuse of notation.

If $\tilde\Lambda$ is a finite flat algebra over $\Lambda$, there is a natural finite map $\tilde\cW := \mathrm{Spf}(\tilde\cW)\stackrel{\mathrm{w}}{\lra} \cW$, and we say that a point $x\in \tilde\cW$ is arithmetic of weight $s$ and character $\varepsilon$ if $\mathrm{w}(x) = \nu_{s,\varepsilon}$. As in the introduction, let $\varepsilon_{\cyc}$ denote the $p$-adic cyclotomic character and $\underline{\varepsilon}_{\cyc}$ be the composition of $\varepsilon_{\cyc}$ with the natural inclusion $\Z_p^\times \subset \Lambda^{\times}$ taking $z$ to the group-like element $[z]$ in $ \Lambda^{\times}$.

Let $N \geq 1$ be an integer not divisible by $p$, and let $\chi: (\mathbb Z/N \mathbb Z)^{\times} \rightarrow \mathbb C_p^{\times}$ be a Dirichlet character.
Let $\hg$ be a Hida family of tame level $N$ and tame character $\chi$ as defined and normalized in \cite[\S 7.2]{KLZ}. Let $\Lambda_{\hg}$ the associated Iwasawa algebra (cf.\,Def.\,7.2.5 of loc.\,cit.), which is finite and flat over $\Lambda$. As in \cite[\S 7.3]{KLZ}, we may specialize $\hg$ at any arithmetic point $x \in \mathcal W_{{\bf g}}=\Spf(\Lambda_{{\bf g}})$ of weight  $k \geq 2$ and character $\varepsilon: (\Z/p^r\Z)^\times \lra L_p^\times$ and obtain
 a classical $p$-ordinary eigenform $g_x$ in the space $S_k(Np^r,\chi \varepsilon)$ of cusp forms of weight $k$, level $Np^r$ and nebentype $\chi \varepsilon$.


\begin{defi}
Let $x \in \mathcal W_{{\bf g}}$ be an arithmetic point of weight $k\geq 1$ and character $\varepsilon$. We say $x$ is crystalline if $\varepsilon=1$  and there exists an eigenform $g_x^\circ$ of level $N$ such that $g_x$ is the ordinary $p$-stabilization of $g_x^\circ$ (which given the previous condition, is automatic if $k>2$, but not necessarily if $k\leq 2$). We denote by $\mathcal W_{{\bf g}}^{\circ}$ the set of crystalline arithmetic points of $\mathcal W_{{\bf g}}$.
\end{defi}

As shown by Wiles \cite{Wi}, \cite[\S 7.2]{KLZ} and already recalled in the introduction, a Hida family $\hg$ as above comes equipped with a free $\Lambda_{{\bf g}}[G_\Q]$-module $\mathbb V_{{\bf g}}$ of rank two, yielding a Galois representation $\varrho_{\hg}: G_{\Q} \lra \GL_2(\Lambda_{\hg})$, all whose classical specializations recover the $p$-adic Galois representation associated to the classical specialization of $\hg$. We assume throughout that the mod $p$ residual representation $\bar\varrho_{\hg}$ associated to any of its classical specializations is irreducible, and that the semi-simplification of $\bar\varrho_{|G_{\Q_p}}$ is non-scalar, as in hypotheses (H1-H2) in the introduction.

Similarly as in \eqref{filtration-classical}, the restriction to $G_{\mathbb Q_p}$ of $\mathbb V_{\hg}$ admits a filtration
\begin{equation}\label{primfil}
0 \rightarrow \mathbb V_{{\bf g}}^+ \rightarrow \mathbb V_{{\bf g}} \rightarrow \mathbb V_{{\bf g}}^- \rightarrow 0,
\end{equation}
where $\mathbb V_{{\bf g}}^+$ and $\mathbb V_{{\bf g}}^-$ are flat $\Lambda_{{\bf g}}[G_{\mathbb Q_p}]$-modules, free of rank one over $\Lambda_{{\bf g}}$. If we let $\psi_{{\bf g}}$ denote the unramified character of $G_{\Q_p}$ taking the arithmetic Frobenius element $\Fr_p$ to $a_p({\bf g})$, then
\begin{equation}\label{vaps}
\mathbb V_{{\bf g}}^+ \simeq \Lambda_{{\bf g}}(\psi_{{\bf g}}^{-1} \chi \varepsilon_{\cyc}^{-1} \underline{\varepsilon}_{\cyc}), \qquad \mathbb V_{{\bf g}}^- \simeq \Lambda_{{\bf g}}(\psi_{{\bf g}}).
\end{equation}

Let $\widehat{\mathbb Z_p^{\ur}}$ be the ring of integers of the completion of the maximal unramified extension of $\mathbb Q_p$. Given a finite-dimensional $G_{\Q_p}$-module $V$ with coefficients over a finite extension $L_p$ of $\mathbb Q_p$, the de Rham Dieudonn\'e module associated to $V$ is defined as \[ D(V)=(V \otimes_{\Q_p} B_{\dR})^{G_{\mathbb Q_p}}, \] where $B_{\dR}$ is Fontaine's field of de Rham periods. If $V$ is unramified there is a further canonical isomorphism \[ D(V) \simeq \Big(V_{\mathrm{int}} \hat \otimes_{\mathbb Z_p} \widehat{\mathbb Z_p^{\ur}} \Big)^{\Fr_p=1}[\frac{1}{p}], \]
where $V_{\mathrm{int}}$ is an integral lattice in $V$.
Similarly, if $\mathbb V$ is a free module over an Iwasawa algebra equipped with an unramified action of $G_{\mathbb Q_p}$, we may define
\begin{equation}\label{dieu}
\mathbb D(\mathbb V):= \Big( \mathbb V \hat \otimes_{\mathbb Z_p} \widehat{\mathbb Z_p^{\ur}} \Big)^{\Fr_p=1}.
\end{equation}

Given a newform $g\in S_k(N,\chi)$ as at the beginning of this section, $D(V_g)$ is an $L_p$-filtered vector space of rank $2$. Set $g^* = g \otimes \bar \chi \in S_k(N, \bar \chi)$. Poincar\'e duality induces a perfect pairing
\begin{equation}\label{Poincare}
\langle \, , \, \rangle: D(V_g(-1)) \times D(V_{g^*}) \rightarrow L_p.
\end{equation}

If $g$ is ordinary at $p$, then \eqref{filtration-classical} gives rise to an exact sequence of Dieudonn\'e modules \[ 0 \rightarrow D(V_g^+) \xrightarrow{i} D(V_g) \xrightarrow{\pi} D(V_g^-) \rightarrow 0 \]
where $D(V_g^+)$ and $D(V_g^-)$ have rank $1$ over $L_p$.

If $k \geq 2$, Faltings' comparison theorem allows to associate to $g$ a regular differential form $\omega_g \in \Fil(D(V_g))$, which induces a linear form \[ \omega_g: D(V^+_{g^*}) \rightarrow L_p, \quad \eta \mapsto \langle \omega_g, \eta \rangle. \]
There is also the differential form $\eta_g$, characterized by the properties that it spans the line $D(V_g^+)$ and $\langle \eta_g,\omega_{g^*}\rangle = 1$. It again gives rise to a linear functional \[ \eta_g: D(V_{g^*}^-) \rightarrow L_p, \quad \omega \mapsto \langle \pi^{-1}(\omega),\eta_g \rangle. \]

As shown in \cite{Oh} and \cite{KLZ}, the differential forms (or linear functionals) $\omega_g$ and $\eta_g$ vary in families. In order to recall this more precisely, let ${\bf g}$ be a Hida family of tame level $N$ and tame character $\chi$ as above. Set $\hg^* = \hg \otimes \bar\chi$. Let $\mathcal Q_{{\bf g}}$ denote the fraction field of $\Lambda_{{\bf g}}$, and set $\mathbb U_{{\bf g}}^+ := \mathbb V_{{\bf g}}^+(\chi^{-1}\varepsilon_{\cyc} \underline{\varepsilon}_{\cyc}^{-1})$; denote by $U_{g_y}^+$ its specialization at a point $y \in \mathcal W_{{\bf g}}^{\circ}$.
By \cite[Proposition 10.1.1]{KLZ}, there exist
\begin{itemize}
\item A homomorphism of $\Lambda_{{\bf g}}$-modules

\begin{equation}\label{omegaG}
 \langle \, , \omega_{{\bf g}} \rangle: \mathbb D(\mathbb U_{{\bf g*}}^+) \rightarrow \Lambda_{{\bf g}}
\end{equation}
  such that for every $y \in \mathcal W_{{\bf g}}^{\circ}$, the specialization of $\omega_{{\bf g}}$ at $y$ is the linear form \[ y \circ \langle \, , \omega_{{\bf g}} \rangle = \langle \, , \mathrm{Pr}^{\alpha *}(\omega_{g^\circ_y}) \rangle: D(U_{g_y^*}^+) \rightarrow L_p \]
where $\mathrm{Pr}^{\alpha *}$ is the $p$-stabilization pull-back isomorphism defined in loc.\,cit.. Note that this makes sense: since $L_p$ is assumed to contain the $N$-th roots of unity, $D(U_{g_y^*}^+)$ and $D(V_{g_y^*}^+)$ are the same module up to a shift in their filtration.

\item  A homomorphism of $\Lambda_{{\bf g}}$-modules
\begin{equation}\label{etaG}
 \langle \, , \eta_{{\bf g}} \rangle: \mathbb D(\mathbb V_{{\bf g*}}^-) \rightarrow \mathcal Q_{{\bf g}}
\end{equation}
such that for every $y \in \mathcal W_{{\bf g}}^{\circ}$ we have
\[ y \circ \langle \, , \eta_{{\bf g}} \rangle = \frac{\mathrm{Pr}^{\alpha *}(\eta_{g_y^{\circ}})}{\lambda_N(g^\circ_y) \mathcal E_0(g_y^{\circ}) \mathcal E_1(g_y^{\circ})}: D(V_{g_y^*}^-) \rightarrow L_p, \] where   $\lambda_N(g^\circ_y)$ stands for the pseudo-eigenvalue of $g_y^\circ$, and \[ \mathcal E_0(g_y^{\circ})=1-\chi^{-1}(p)\beta_{g^\circ_y}^2p^{1-k}, \quad \mathcal E_1(g_y^{\circ})=1-\chi(p)\alpha_{g^\circ_y}^{-2}p^{k-2}. \]
\end{itemize}

\subsection{Hida--Rankin's three-variable $p$-adic $L$-function}\label{HidaRankin}

Let ${\bf g} \in \Lambda_{\hg}[[q]]$, ${\bf h}\in \Lambda_{\hh}[[q]]$ be a pair of Hida families of tame level $N$ and tame characters $\chi_g$ and $\chi_h$ respectively. As in the introduction, set $\Lambda_{{\bf g h}}:=\Lambda_{{\bf g}} \hat \otimes_{\mathbb Z_p} \Lambda_{{\bf h}} \hat \otimes_{\mathbb Z_p} \Lambda$ and $\mathcal W_{{\bf gh}}:= \Spf(\Lambda_{{\bf gh}})$.  Let $\mathcal Q_{{\bf gh}}$ denote the fraction field of $\Lambda_{{\bf gh}}$, and set $\mathcal W_{{\bf gh}}^\circ:= \cW_\hg^\circ \times \cW_\hh^\circ \times \cW^\cl$.

\begin{defi} The {\em critical} range  is the set of points $(y,z,\sigma) \in \mathcal W_{{\bf gh}}^\circ$ of weights $(\ell,m,s)$ such that $\ell,m \geq 2$ and $m \leq s < \ell$.

The {\em geometric} range is defined to be the set of points $(y,z,\sigma) \in \mathcal W_{{\bf gh}}^\circ$ of weights $(\ell,m,s)$ such that $\ell,m \geq 2$ and $1 \leq s < \min(\ell,m)$.
\end{defi}

Hida constructed \cite{Hi2}, \cite{Hi3} a three-variable $p$-adic Rankin $L$-function $L_p({\bf g},{\bf h})$ on $\mathcal W_{{\bf gh}}$, interpolating the algebraic parts of the critical values $L(g^\circ_y,h^\circ_z,s)$ for every triple of classical points $(y,z,s)$ in $\mathcal W_{{\bf gh}}^{^\circ}$ lying in the {\em critical} range. More precisely, \cite[Theorem 5.1d]{Hi3} asserts the following.

\begin{theorem}(Hida)\label{Hida-3var}
There exists a unique element $L_p({\bf g},{\bf h}) \in \mathcal Q_{{\bf gh}}$ whose value at any $(y,z,s)\in\mathcal W_{{\bf gh}}^\circ$ of weights $(\ell,m,s)$ in the {\em critical} range is well-defined and equal to
$$
L_p({\bf g},{\bf h})(y,z,s) = \frac{C \cdot \mathcal E(y,z,s)}{(2\pi i)^{2s-m+1} \langle g_y^\circ,g_y^\circ \rangle}  \times L(g_y^\circ,h_z^\circ,s)
$$
where $C$ is a non-zero algebraic number in the finite extension $\mathbb Q(g_y^{\circ},h_z^{\circ})$ generated by the Fourier coefficients of $g_y^{\circ}$ and $h_z^{\circ}$, $\langle g_y^{\circ}, g_y^{\circ} \rangle$ is the Petersson norm as normalized in loc.\,cit., and
\begin{equation}\label{Euler}
 \mathcal E(y,z,s)= \Big(1-\frac{p^{s-1}}{\alpha_{g^\circ_y} \alpha_{h^\circ_z}} \Big) \Big(1-\frac{p^{s-1}}{\alpha_{g^\circ_y} \beta_{h^\circ_z}} \Big) \Big(1-\frac{\beta_{g^\circ_y} \alpha_{h^\circ_z}}{p^{s}}\Big) \Big(1-\frac{\beta_{g^\circ_y} \beta_{h^\circ_z}}{p^{s}}\Big).
\end{equation}
\end{theorem}

Let $g$ and $h$ be classical specializations of the families $\hg$ and $\hh$ at some point $(y_0,z_0)\in \cW_{\hg}^\cl \times \cW_{\hh}^\cl$ of weight $(\ell,m)$.
We denote by $L_p(g, h,s)$ the restriction of $L_p({\bf g},{\bf h})(y,z,s)$ to the line $(y_0,z_0,s)$. As quoted e.g. in \cite[\S9.2]{Das3}, this $p$-adic $L$-function satisfies the functional equation
\begin{equation}\label{functional}
 L_p(g, h, \ell+m-1-s) = \epsilon(g,h,s) L_p(g^*, h^*, s)
\end{equation}
where $\epsilon(g,h,\psi,s) = A \cdot B^s$, with $A\in \mathbb Q(g_y,h_z)^\times $ and $B\in \Q^\times$.

\subsection{Improved $p$-adic $L$-functions}

As before, let ${\bf g} \in \Lambda_{\hg}[[q]]$ be a Hida family of tame level $N$ and tame character $\chi=\chi_g$, and set ${\bf h} = {\bf g^*} := {\bf g} \otimes \chi^{-1}$. As in the introduction, define the surface
\begin{equation}\label{S}
\mathcal S_{\mathrm{Hida}} := \mathcal S_{\ell,m,m} := \{ (y,z,\sigma) \in \mathcal W_{{\bf g}} \times \mathcal W_{{\bf g}} \times \mathcal W: \, \mathrm{w}(z)=\sigma \}.
\end{equation}
Note that this is a sub-variety of $ \mathcal W_{\hg \hg^*}$ all whose crystalline arithmetic points have weights $(\ell,m,m)$ for some $\ell,m\geq 1$.

The restriction to $\mathcal S_{\mathrm{Hida}}$ of the {\em second multiplier} in the Euler-like factor \eqref{Euler} appearing in the interpolation formula for Hida--Rankin's $p$-adic $L$-function is \[ 1-\frac{p^{s-1}}{\alpha_{g_y^{\circ}} \beta_{h_z^{\circ}}} = 1- \frac{\alpha_{g_z^{\circ}}}{\alpha_{g_y^{\circ}}}. \] This expression interpolates to an Iwasawa function in $\Lambda_{{\bf g}} \times \Lambda_{{\bf g}}$, which by abuse of notation we continue to denote with the same symbol.  One naturally expects $1- \frac{\alpha_{g_z^{\circ}}}{\alpha_{g_y^{\circ}}}$ should divide the two-variable $p$-adic $L$-function $L_p({\bf g},{\bf g^*})(y,z,s)$, where $(y,z)$ vary in $\mathcal W_{{\bf g}} \times \mathcal W_{{\bf g}}$ and $s=\mathrm{w}(z)$ is determined by the weight of $z$. Hida proved in \cite{Hi3} the following stronger statement:

\begin{theorem}[Hida]\label{improved}
Let $\hat L_p({\bf g},{\bf g^*})$ be the unique element in the fraction field of $\Lambda_{{\bf g}} \hat \otimes \Lambda_{{\bf g}}$ such that \[ L_p({\bf g},{\bf g^*})(y,z,\mathrm{w}(z)) = \Big(1-\frac{\alpha_{g_z^{\circ}}}{\alpha_{g_y^{\circ}}} \Big) \cdot \hat L_p({\bf g},{\bf g^*})(y,z). \] Then, fixing $y_0 \in \mathcal W_{{\bf g}}^\circ$, the one-variable meromorphic function \[ \hat L_p({\bf g},{\bf g^*})(y_0,z) \] on $\mathcal W_{{\bf g}}$ has a simple pole at $z=y_0$ whose residue is a non-zero explicit rational number.
\end{theorem}

Associated to the adjoint representation attached to any classical specialization $g=g_{y_0}$ of the Hida family $\hg$ at some arithmetic point $y_0\in \cW_{\hg}^{\cl}$, Hida defined an analytic $\cL$-invariant, which can be recast in several equivalent ways (cf.\,the works of Hida, Harron, Citro and Dasgupta (cf.\,\cite{Hi-ad}, \cite{Ci}, \cite{Das3}). We may define it for instance as:
\begin{equation}\label{citro}
\mathcal L(\ad^0(g_{y_0})) :=  \frac{-\alpha'_{\hg}(y_0)}{\alpha_{\hg}(y_0)},
\end{equation}
where recall $\alpha_{\hg} = a_p(\hg) \in \Lambda_{\hg}$ is the Iwasawa function given by the eigenvalue of the Hecke operator $U_p$ acting on $\hg$, and $\alpha'_{\hg}$ is its derivative.

Let $L_p'(\ad^0(g_{y_0}),s)$ denote Hida-Schmidt's $p$-adic $L$-function associated to the adjoint of the ordinary eigenform $g_{y_0}$ (cf.\,\cite{Sc}, \cite{Hi-ad}). The argument below is mainly due to Citro and Dasgupta, but since in loc.\,cit.\,they often assume that $\ell \geq 2$, we include it  in order to ensure that it holds as well at weight $1$, which is the case we mostly focus on. The main  point is that the objects in play all vary in Hida families.

\begin{propo}\label{impad}
For a crystalline classical point $y_0\in \cW_{\hg}^\circ$ of weight $\ell \geq 1$, we have
\[  \mathcal L(\ad^0(g_{y_0})) = L_p({\bf g},{\bf g^*})(y_0,y_0,\ell) = L_p'(\ad^0(g_{y_0}),\ell), \]
up to a non-zero rational constant.
\end{propo}
\begin{proof}
The first equality follows from Theorem \ref{improved}, which amounts to say that \[ L_p({\bf g},{\bf g^*})(y_0,y_0,\ell) = \lim_{z \rightarrow y_0} \Big(1-\frac{\alpha_{g_z^{\circ}}}{\alpha_{g_{y_0}^{\circ}}} \Big) \cdot \hat L_p({\bf g},{\bf g^*})(y_0,z). \] Since $\Big(1-\frac{\alpha_{g_z}^{\circ}}{\alpha_{g_{y_0}^{\circ}}} \Big)$ vanishes at $z=y_0$ and $\hat L_p(\hg,\hg^*)(y_0,z)$ has a pole at $z=y_0$ given by a non-zero rational number, the value of the previous limit agrees, modulo $L^{\times}$, with the derivative of the first factor, i.e., \[ L_p({\bf g},{\bf g^*})(y_0,y_0,\ell) = \frac{-\alpha'_{\hg}(y_0)}{\alpha_{\hg}(y_0)}. \]

In addition, Dasgupta's factorization proved in \cite{Das3} asserts that $L_p({\bf g},{\bf g^*})(y_0,y_0,s) = \zeta_p(s-\ell+1) L_p(\ad^0(g_{y_0}),s)$. Here $\zeta_p(s)$ is the $p$-adic zeta function, which has a pole at $s=1$ with non-zero rational residue. The second factor vanishes at $s=\ell$ and it follows that $L_p({\bf g},{\bf g^*})(y_0,y_0,\ell) = L'_p(\ad^0(g_{y_0}),\ell) \, (\mathrm{mod} \,\Q^\times)$.
\end{proof}

\section{Derived Beilinson--Flach elements}\label{sec-derived}

\subsection{The three-variable Euler system of Kings, Lei, Loeffler and Zerbes}\label{memo}

Let $\hg$ and $\hh$ be a pair of $p$-adic cuspidal Hida families of tame conductor $N$ and tame nebentype $\chi_g$ and $\chi_h$ as in \S \ref{Hidafam}.
As in the Introduction, and keeping the notations of the previous section, define the $\Lambda_{{\bf gh}}$-module
\begin{equation}\label{Vgh}
\mathbb V_{{\bf gh}}:=\mathbb V_{{\bf g}} \hat \otimes_{\mathbb Z_p} \mathbb V_{{\bf h}} \hat \otimes_{\mathbb Z_p} \Lambda(\varepsilon_{\cyc} \underline{\varepsilon}_{\cyc}^{-1}).
\end{equation}

This $\Lambda$-adic Galois representation is characterized by the property that for any $(y,z,\sigma) \in \cW_{{\bf gh}}^\circ$ with $\mathrm{w}(\sigma) = \nu_s$ with $s\in \Z$, \eqref{Vgh} specializes to
$$
\mathbb V_{\hg \hh}(y,z,\sigma) = V_{g_y} \otimes V_{h_z}(1-s),
$$
the $(1-s)$-th Tate twist of the tensor product of the Galois representations attached to $g_y$ and $h_z$.

Fix $c \in \Z_{>1}$ such that $(c,6pN_gN_h)=1$. \cite[Theorem A]{KLZ} yields a three-variable $\Lambda$-adic global Galois cohomology class
$$
\kappa^c({\bf g},{\bf h}) \in H^1(\mathbb Q, \mathbb V_{{\bf gh}})
$$
that is referred to as the Euler system of Beilinson--Flach elements associated to $\hg$ and $\hh$. We denote by $\kappa_p^c(\hg,\hh) \in H^1(\mathbb Q_p, \mathbb V_{\hg \hh})$ the image of $\kappa^c(\hg,\hh)$ under the restriction map.

Since $c$ is fixed throughout, we may sometimes drop it from the notation. This constant does make an appearance in fudge factors accounting for the interpolation properties satisfied by the Euler system, but in all cases we are interested in these fudge factors do not vanish and hence do not pose any problem for our purposes.

Given a crystalline arithmetic point $(y,z,s) \in \cW_{{\bf gh}}^\circ$ of weights $(\ell,m,s)$, set for notational simplicity throughout this section $g=g_y^\circ$, $h=h_z^\circ$. With these notations, $g_y$ (resp.\,$h_z$) is the $p$-stabilization of $g$ (resp.\,$h$) with $U_p$-eigenvalue $\alpha_g$ (resp.\,$\alpha_h$).

Define
\begin{equation}\label{home}
\kappa(g_y,h_z,s):=\kappa({\bf g}, {\bf h})(y,z,s) \in H^1(\mathbb Q, V_{g_y} \otimes V_{h_z}(1-s))
\end{equation}
as the specialisation of $\kappa({\bf g},{\bf h})$ at $(y,z,s)$.

If one further assumes that $(y,z,s)$ lies in the geometric range, Kings, Loeffler and Zerbes showed in \cite{KLZ1} that the cohomology group appearing in \eqref{home} also hosts a canonical {\em Rankin-Eisenstein class}, denoted
\begin{equation}\label{Eisenstein}
\Eis_{\et}^{[g,h,s]} \in H^1(\mathbb Q, V_{g} \otimes V_{h}(1-s)).
\end{equation}
This class is attached to the classical pair $(g,h)$ and can be constructed purely by geometric methods, without appealing to the variation of $(g_y,h_z)$ in $p$-adic families. It is for this reason that in fact the classes $\Eis_{\et}^{[g,h,s]}$ lie in the Bloch--Kato Selmer subgroup
$$
H^1_{\fin}(\mathbb Q, V_{g} \otimes V_{h}(1-s)) \subset H^1(\mathbb Q, V_{g} \otimes V_{h}(1-s)).
$$
We refer to \cite[\S 5]{KLZ1} and \cite[Definition 3.3.2]{KLZ} for the precise statements.

Since both $ \kappa(g_y,h_z,s)$ and $\Eis_{\et}^{[g,h,s]}$ live in the same space, it  makes sense to ask whether they are related. This is the content of \cite[Theorem A (8.1.3)]{KLZ}:

\begin{theorem}\label{KLZ-ThmA}
Assume $(y,z,s) \in \mathcal W_{{\bf gh}}^\circ$ lies in the geometric range. Then
\begin{equation}\label{rel}
\kappa(g_y,h_z,s) = \mathcal E(g,h,s) \cdot \Eis_{\et}^{[g,h,s]}
\end{equation}
where
\begin{equation}\label{unfac}
\mathcal E(g,h,s) = \frac{\Big(1-\frac{p^{s-1}}{\alpha_{g}\alpha_{h}}\Big)\Big(1-\frac{\alpha_{g} \beta_{h}}{p^s}\Big)\Big(1-\frac{\beta_{g} \alpha_{h}}{p^s}\Big)\Big(1-\frac{\beta_{g} \beta_{h}}{p^s}\Big)(c^2-c^{2s-\ell-m+2})}{(-1)^{s-1} (s-1)! \binom{\ell-2}{s-1} \binom{m-2}{s-1}}.
\end{equation}
In particular $\kappa(g_y,h_z,s)$ lies  in $H_{\fin}^1(\mathbb Q, V_{g} \otimes V_{h}(1-s))$.
\end{theorem}

The following proposition recalls the existence of the so-called Perrin-Riou big logarithm, interpolating the Bloch--Kato logarithm $\log_{\mathrm{BK}}$ and dual exponential map $\exp^*_{\mathrm{BK}}$ associated to the classical specializations of a $\Lambda$-adic representation of $G_{\Q_p}$. We refer to \cite{BK} and \cite{Bel} for an introduction to $p$-adic Hodge theory and the definitions of these maps.

Recall the unramified character $\psi_{{\bf g}}$ of $G_{\mathbb Q_p}$ taking a Frobenius element $\Fr_p$ to $a_p({\bf g})$, and as before, let $\varepsilon_{{\bf g}}$ be the composition of the cyclotomic character $\underline{\varepsilon}_{\cyc}$ with the natural inclusion $\Lambda^{\times} \subset \Lambda_{{\bf g}}^{\times}$. Define the $G_{\Q_p}$-subquotient $$\mathbb V_{{\bf gh}}^{-+} := \mathbb V_{{\bf g}}^- \hat \otimes \mathbb V_{{\bf h}}^+$$ of $\mathbb V_{{\bf g}} \hat \otimes \mathbb V_{{\bf h}}$ of rank one over the two-variable Iwasawa algebra $\Lambda_{\hg} \hat\otimes \Lambda_{\hh}$. In light of \eqref{vaps}, the Galois action on $\mathbb V_{{\bf gh}}^{-+}$ is given by the character \begin{equation}\label{char-+} \eta_{{\bf h}}^{{\bf g}}:=\varepsilon_{\cyc}^{-1} \chi_h \cdot \psi_{{\bf g}} \otimes \psi_{{\bf h}}^{-1}\varepsilon_{{\bf h}}. \end{equation}

It follows that $\mathbb U_{{\bf gh}}^{-+}:=  \mathbb V_{{\bf gh}}^{-+}(\varepsilon_{\cyc} \varepsilon_{{\bf h}}^{-1})$ is an unramified $G_{\Q_p}$-module and we can thus invoke its $\Lambda$-adic Dieudonn\'e module as defined in \eqref{dieu}.

\begin{propo}\cite[Theorem 8.2.8]{KLZ}\label{biglog}
There is an injective morphism of $\Lambda_{{\bf gh}}$-modules \[ \mathcal L^{-+}_{{\bf gh}}: H^1(\mathbb Q_p, \mathbb V_{{\bf gh}}^{-+} \hat \otimes \Lambda(\varepsilon_{\cyc} \underline{\varepsilon}_{\cyc}^{-1})) \rightarrow \mathbb D(\mathbb U_{{\bf gh}}^{-+}) \hat \otimes \Lambda \]
such that for all $\kappa_p \in  H^1(\mathbb Q_p, \mathbb V_{{\bf gh}}^{-+} \hat \otimes \Lambda(\varepsilon_{\cyc} \underline{\varepsilon}_{\cyc}^{-1}))$ and all $(y,z,s) \in \mathcal W_{{\bf gh}}^\circ$ of weights $(\ell,m,s)$:

\begin{itemize}
\item if $s<m$, $(\mathcal L^{-+}_{{\bf gh}}(\kappa_p))_{y,z,s}=\Big( 1-\frac{p^{s-1}}{\alpha_{g} \beta_{h}} \Big) \Big(1-\frac{\alpha_{g} \beta_{h}}{p^{s}} \Big)^{-1} \cdot \frac{(-1)^{m-s+1}}{(m-s+1)!} \cdot \log_{\BK}(\kappa_p(y,z,s))$;
\item if $s \geq m$, $(\mathcal L^{-+}_{{\bf gh}}(\kappa_p))_{y,z,s} = \Big( 1-\frac{p^{s-1}}{\alpha_{g} \beta_{h}} \Big) \Big(1-\frac{\alpha_{g} \beta_{h}}{p^{s}} \Big)^{-1} (s-m)! \cdot \exp_{\BK}^*(\kappa_p(y,z,s))$.
\end{itemize}

Here, $\log_{\BK}$ and $\exp_{\BK}^*$ stand for the Bloch--Kato logarithm (resp. dual exponential) associated to the Dieudonn\'e module of the $p$-adic representation $V_{g_y}^- \otimes U_{h_z}^+(1-s)$.
\end{propo}

\begin{remark}
In the case where $\alpha_{g} \beta_{h} = p^s$, we implicitly understand that the Euler factor in the denominator appears on the left hand side of the equality.
\end{remark}

As shown in \cite[Theorem 8.1.7]{KLZ}, there is an injection \[ H^1(\mathbb Q_p, \mathbb V_{{\bf g}}^- \hat \otimes \mathbb V_{{\bf h}}^+ \hat \otimes \Lambda (\varepsilon_{\cyc} \underline{\varepsilon}_{\cyc}^{-1})) \hookrightarrow H^1(\mathbb Q_p, \mathbb V_{{\bf g}}^- \hat \otimes \mathbb V_{{\bf h}} \hat \otimes \Lambda (\varepsilon_{\cyc} \underline{\varepsilon}_{\cyc}^{-1})). \] If we denote by $\kappa_p^{--}({\bf g},{\bf h})$ the projection of  $\kappa_p({\bf g},{\bf h})$ to $H^1(\mathbb Q_p, \mathbb V_{{\bf g}}^- \hat \otimes \mathbb V_{{\bf h}}^- \hat \otimes \Lambda (\varepsilon_{\cyc} \underline{\varepsilon}_{\cyc}^{-1}))$, it is further shown in loc.\,cit.\,that
\begin{equation}\label{menos}
\kappa_p^{--}({\bf g},{\bf h}) = 0.
\end{equation}

As a consequence, the projection of $\kappa_p({\bf g}, {\bf h})$ to $H^1(\mathbb Q_p, \mathbb V_{{\bf g}}^- \hat \otimes \mathbb V_{{\bf h}} \hat \otimes \Lambda (\varepsilon_{\cyc} \underline{\varepsilon}_{\cyc}^{-1}))$ actually lies in $H^1(\mathbb Q_p, \mathbb V_{{\bf g}}^- \hat \otimes \mathbb V_{{\bf h}}^+ \hat \otimes \Lambda(\varepsilon_{\cyc} \underline{\varepsilon}_{\cyc}^{-1}))$ and we may hence denote it $\kappa_p^{-+}({\bf g}, {\bf h})$.

Let $\lambda_N({\bf g})$ denote the $\Lambda$-adic pseudo-eigenvalue of $\hg$ as defined in \cite[\S 10]{KLZ}, interpolating the Atkin-Lehner pseudo-eigenvalues of the classical specializations of $\hg$.
Recall from \eqref{omegaG} and \eqref{etaG} Ohta's families of differential forms $\eta_{{\bf g}} \in \mathbb D(\mathbb U_{{\bf g}}^+)$ and $\omega_{{\bf h}} \in \mathbb D(\mathbb V_{{\bf h}}^-)$.
As it follows from the properties recalled in loc.\,cit.,\,
there exists a homomorphism of $\Lambda_{{\bf gh}}$-modules \[ \langle \, , \eta_{{\bf g}} \otimes \omega_{{\bf h}} \rangle: \mathbb D(\mathbb U_{{\bf gh}}^{-+}) \hat \otimes \Lambda \rightarrow \mathcal Q_{{\bf gh}} \otimes \Q_p(\mu_N) \] such that for all $\delta \in \mathbb D(\mathbb U_{{\bf gh}}^{-+})\hat \otimes \Lambda$ and all $(y,z,s) \in \mathcal W_{{\bf gh}}^\circ$,
\begin{eqnarray}\label{pardif}
\nu_{y,z,s}(\langle \delta,\eta_{{\bf g}} \otimes \omega_{{\bf h}} \rangle) = & \frac{1}{\lambda_N(g) \mathcal E_0(g) \mathcal E_1(g)} \cdot \langle \nu_{y,z,s}(\delta),\mathrm{Pr}^{\alpha *}(\eta_{g^\circ_y}) \otimes \mathrm{Pr}^{\alpha *}(\omega_{h^\circ_z)} \rangle \\
\nonumber  = &  \frac{1}{\lambda_N(g) \mathcal E_0(g) \mathcal E_1(g)} \cdot \langle \mathrm{Pr}^{\alpha}_*(\nu_{y,z,s}(\delta)),\eta_{g^\circ_y} \otimes \omega_{h^\circ_z} \rangle,
\end{eqnarray}
where, recall again,
\[ \mathcal E_0(g)=1-\chi_g^{-1}(p) \beta_{g}^2 p^{1-\ell}, \qquad \mathcal E_1(g)=1-\chi_g(p) \alpha_{g}^{-2} p^{\ell-2}. \]

The following explicit reciprocity law is \cite[Theorem 10.2.2]{KLZ}.

\begin{theorem}\label{reclaw} Define the Iwasawa function
\begin{equation}\label{eulfac}
\mathcal A({\bf g}, {\bf h}) := \lambda_N({\bf g})^{-1} (-1)^{s}(c^2-c^{-(\ell+m-2-2s)} \varepsilon_{{\bf g}}(c)^{-1}\varepsilon_{{\bf h}}(c)^{-1})
\end{equation}
 in $\Lambda_{\hg \hh}$. Then
\begin{equation}\label{reclawf}
\langle \mathcal L_{{\bf gh}}^{-+}(\kappa^{-+}_p({\bf g}, {\bf h})), \eta_{{\bf g}} \otimes \omega_{{\bf h}} \rangle = \mathcal A({\bf g}, {\bf h}) \cdot L_p({\bf g}, {\bf h}).
\end{equation}
\end{theorem}

\subsection{The self-dual case}\label{self-dual}

Let $\hg$ be a Hida family of tame conductor $N$ and tame nebentype $\chi$, and set again $\hh=\hg^* = \hg \otimes \bar{\chi}$. Define the curve
\begin{equation}\label{C}
\cC := \cC_{\ell,\ell,\ell-1} = \{(y,z,\sigma) \in \cW_{\hg \hg^*}: \,  y=z, \quad \alpha_{g_y} \neq \beta_{g_y}, \quad \mathrm{w}(z)=\sigma \cdot \varepsilon_{\cyc}\}.
\end{equation}

Note that $\cC$ is a finite cover of the line in $\cW^3$ given as the set of regular points in the Zariski closure of the set of points of weights $(\ell,\ell,\ell-1)$ for some $\ell \geq 1$.

\begin{theorem}\label{classes-zero}
The restriction of $\kappa_p(\hg,\hg^*)$ to $\cC$ is zero.
\end{theorem}

\begin{proof} Recall firstly from \eqref{menos} that $\kappa_p^{--}({\bf g}, {\bf g^*})=0$. We also claim that $\kappa_p^{-+}(\hg,\hg^*)_{|\cC}= \kappa_p^{+-}(\hg,\hg^*)_{|\cC} = 0$. To see this, observe that the $\Lambda_{\hg}[G_{\Q_p}]$-module $\mathbb V_{\hg \hg^*}^{-+} \hat \otimes \Lambda(\varepsilon_{\cyc} \underline{\varepsilon}_{\cyc}^{-1})_{|\cC}$ is isomorphic to $\Lambda_{\hg}(1)$. This follows directly from \eqref{char-+} and \eqref{C}, because $\frac{\alpha_{g_y^{\circ}}}{\alpha_{g_y^{\circ}}} \cdot  \varepsilon_{\cyc}^{\ell-1} \cdot \varepsilon_{\cyc}^{2-\ell} = \varepsilon_{\cyc}$. Hence
\begin{equation}\label{two-comp}
H^1(\mathbb Q_p, \mathbb V_{{\bf gg^*}}^{-+} \hat \otimes \Lambda(\varepsilon_{\cyc} \underline{\varepsilon}_{\cyc}^{-1})_{|\cC}) \simeq H^1(\mathbb Q_p, \Lambda_{\hg}(1)) \simeq H^1(\mathbb Q_p, \mathbb Z_p(1)) \hat \otimes \Lambda_{\hg}
\stackrel{(\ord_p,\log_p)}{\simeq} \Lambda_{\hg} \oplus \Lambda_{\hg}
\end{equation}
and $\kappa_p^{-+}(\hg,\hg^*)_{|\cC}$ vanishes if and only if infinitely many of its specializations are zero. But this is true for any crystalline classical point $(y,y,\ell-1)$ on $\cC$ with $\ell>1$. Indeed, the factor $\mathcal E(g_y,h_z,s)$ of Theorem \ref{KLZ-ThmA} vanishes, as $\alpha_{g} \beta_{g}= \chi(p)p^{\ell-1}$, and hence
\begin{equation}\label{spec-eulfac}
\alpha_{h} = \alpha_{g} \cdot \chi^{-1}(p) = \frac{\alpha_{g} p^{\ell-1}}{\alpha_{g} \beta_{g}} = \frac{p^{\ell-1}}{\beta_{g}}.
\end{equation}
By \eqref{rel} this shows that the specialization of the global cohomology class $\kappa(\hg,\hg^*)(y,y,\ell-1)$ is zero for all $y\in \cW_{\hg}^\circ$ of weight $\ell>1$, and a fortiori $\kappa_p^{-+}(\hg,\hg^*)(y,y,\ell-1)=0$. We conclude that $\kappa_p^{-+}(\hg,\hg^*)_{|\cC}=0$ and likewise $\kappa_p^{+-}(\hg,\hg^*)_{|\cC}=0$ by a symmetric reasoning.





Finally, note that there is an exact sequence
\begin{equation}\label{ex-se}
 0 \ra H^1(\mathbb Q_p, \mathbb V_{{\bf gg^*}}^{++} \hat \otimes \Lambda(\varepsilon_{\cyc} \underline{\varepsilon}_{\cyc}^{-1})_{|\cC}) \rightarrow H^1(\mathbb Q_p, \mathbb V_{{\bf gg^*}|\cC}) \rightarrow H^1(\mathbb Q_p, \mathbb V_{\hg \hg^*}/(\mathbb V_{\hg \hg^*}^{++}\hat \otimes \Lambda(\varepsilon_{\cyc} \underline{\varepsilon}_{\cyc}^{-1}))_{|\cC}).
 \end{equation}

The first map above is injective because $H^0(\mathbb Q_p, \mathbb V_{\hg \hg^*}/(\mathbb V_{\hg \hg^*}^{++}\hat \otimes \Lambda(\varepsilon_{\cyc} \underline{\varepsilon}_{\cyc}^{-1}))_{|\cC})=0$, as it follows again from the description of $\mathbb V_{\hg \hg^*}^{-+}$ given in \eqref{char-+} (and similarly for $\mathbb V_{\hg \hg^*}^{+-}$ and $\mathbb V_{\hg \hg^*}^{--}$).

Since we have already shown that $\kappa_p^{--}(\hg,\hg^*)_{|\cC}= \kappa_p^{-+}(\hg,\hg^*)_{|\cC}= \kappa_p^{+-}(\hg,\hg^*)_{|\cC} = 0$, this implies that the image of $\kappa(\hg,\hg^*)_{|\cC}$ in the right-most term of \eqref{ex-se} vanishes. Hence, $\kappa(\hg,\hg^*)_{|\cC}$ lies in $H^1(\mathbb Q_p, \mathbb V_{\hg}^+ \hat \otimes \mathbb V_{\hg^*}^+ \hat \otimes \Lambda(\varepsilon_{\cyc} \underline{\varepsilon}_{\cyc}^{-1})_{|\cC})$. It follows from \cite[Theorem 8.2.3, Remark 8.2.4]{KLZ} that  the latter space is isomorphic to $\Lambda_{\hg}$, and thus $\kappa(\hg,\hg^*)_{|\cC}$ is zero if and only if infinitely many of its specializations are, which is the case as already argued above.

\end{proof}


\subsection{A derived system of Beilinson--Flach elements}\label{deri}

Keep the notations and assumptions as in previous sections. Theorem \ref{classes-zero} above establishes the vanishing of the local cohomology class $\kappa_p(\hg,\hg^*)$ along $\cC$ and it is thus natural to ask about the existence of  a {\it derived} cohomology class $\kappa_p'({\bf g}, {\bf g^*})$ on a proper subspace of three-dimensional weight space $\cW_{\hg \hg^*}$ containing $\cC$, bearing a reciprocity law with Hida--Rankin's improved $p$-adic $L$-function.

The purpose of this section is making this construction explicit. Consider the surface
$$
\cS := \mathcal S_{\ell,m,m-1} = \{(y,z,\sigma) \in \cW_{\hg \hg^*}: \, \mathrm{w}(z)=\sigma \cdot \varepsilon_{\cyc}\},
$$
which is a finite cover of the plane in $\cW^3$
arising as the Zariski closure of points of weights $(\ell,m,m-1)$ for some $\ell, m \geq 1$. Note that $\cS$ obviously contains the curve $\cC$.

Let $\mathbb V_{{\bf gg^*}|\cS}$ denote the restriction of $\mathbb V_{{\bf gg^*}}$ to the surface $\cS$ and $\mathbb V_{{\bf gg^*}|\cC}$ denote its restriction to $\cC$. The following proposition establishes the existence of a class $\kappa_p'(\hg,\hg^*) \in H^1(\mathbb Q_p, \mathbb V_{\hg \hg^*|\cS})$ that may be regarded as the derivative of $\kappa_p(\hg,\hg^*)$ along the $z$-direction.


We shrink weight space $\cW$ to a rigid-analytic  open disk $\cU \subset \cW$ centered at $1$ at which the finite cover $\mathrm{w}: \cW_{\hg} \ra \cW$ restricts to an isomorphism $\mathrm{w}: \cU_{\hg} \stackrel{\sim}{\ra} \cU$ with $y_0\in \cU_{\hg}$. Let $\Lambda_{\cU_{\hg}} = \cO(\cU_{\hg})$ denote the Iwasawa algebra of analytic functions on $\cU_{\hg}$ whose supremum norm is bounded by $1$. Shrink likewise $\cC$ and $\cS$ so that projection to weight space restricts to an isomorphism with $\cU$ and $\cU \times \cU$ respectively. Having done that, their associated Iwasawa algebras are respectively $\cO(\cC) = \Lambda_{\cU_{\hg}} \simeq \Z_p[[Z]]$ and $\cO(\cS) = \Lambda_{\cU_{\hg}}\hat\otimes \Lambda_{\cU_{\hg}}  \simeq \Z_p[[Y,Z]]$. The isomorphism $\Lambda_{\cU_{\hg}} \simeq \Z_p[[Z]]$ is not canonical and depends on the choice of an element $\gamma \in \Lambda_{\cU_{\hg}}^\times$ which is sent to $1+Z$.


Consider the short exact sequence of $\mathbb Z_p$-modules \[ 0 \rightarrow \mathbb Z_p[[Y,Z]] \xrightarrow{\cdot(Z-Y)} \mathbb Z_p[[Y,Z]] \rightarrow \mathbb Z_p[[Z]] \rightarrow 0. \] Under the above identifications
the previous exact sequence may be recast as
\[ 0 \rightarrow \cO_{\cS} \xrightarrow{\delta} \cO_{\cS} \rightarrow \cO_{\cC} \rightarrow 0 \]
with $\delta = 1\otimes(\gamma-1)-(\gamma-1)\otimes 1$ in $ \cO_{\cS} \simeq \Lambda_{\cU_{\hg}}\hat\otimes \Lambda_{\cU_{\hg}}$.

\begin{propo}\label{claseder}
There exists a  unique local class $ \kappa_{p,\gamma}'({\bf g}, {\bf g^*}) \in H^1(\mathbb Q_p, \mathbb V_{{\bf gg^*}|\cS})$ such that \[ \kappa_p({\bf g}, {\bf g^*})_{|\cS} = \delta  \cdot \kappa_{p,\gamma}'({\bf g}, {\bf g^*}). \]
\end{propo}
\begin{proof}
The short exact sequence of $G_{\Q_p}$-modules  \[ 0 \rightarrow \mathbb V_{{\bf gg^*}|\cS} \xrightarrow{\delta} \mathbb V_{{\bf gg^*}|\cS} \rightarrow \mathbb V_{{\bf gg^*}|\cC} \rightarrow 0 \] gives rise to the long exact sequence \[ H^0(\mathbb Q_p, \mathbb V_{{\bf gg^*}|\cC}) \rightarrow H^1(\mathbb Q_p, \mathbb V_{{\bf gg^*}|\cS}) \xrightarrow{\delta} H^1(\mathbb Q_p, \mathbb V_{{\bf gg^*}|\cS}) \rightarrow H^1(\mathbb Q_p, \mathbb V_{{\bf gg^*}|\cC}). \] Since $H^0(\Q_p,\mathbb V_{{\bf gg^*}|\cC})=0$ as already argued in the proof of Theorem \ref{classes-zero}, the vanishing of $\kappa_p({\bf g},{\bf g^*})_{|\cC}$ implies the existence of a unique element $ \kappa_{p,\gamma}'({\bf g}, {\bf g^*}) \in H^1(\mathbb Q_p, \mathbb V_{{\bf gg^*}|\cS})$ satisfying the claim.
\end{proof}

We are interested in the restriction of $\kappa_{p,\gamma}'(\hg,\hg^*)$ to $\cC$; although it depends on the choice of the topological generator $\gamma$, the class $\kappa_p'(\hg,\hg^*) \in H^1(\mathbb Q_p, \mathbb V_{\hg \hg^*|\cC})$ defined by \[ \kappa_p'(\hg,\hg^*) = \frac{\kappa_{p,\gamma}'(\hg,\hg^*)}{\log_p(\gamma)} \] is independent of $\gamma$.

Since in our setting the Euler factor $1-\frac{\alpha_{g_y^{\circ}} \beta_{h_z^{\circ}}}{p^s}$ appearing in Proposition \ref{biglog} is equal to $1-\frac{\alpha_{g_y^{\circ}}}{\alpha_{g_z^{\circ}}}$, it is natural to introduce a modified $p$-adic $L$-function on the surface $\cS$, defined as
\begin{equation}\label{exp1}
\tilde L_p({\bf g}, {\bf g^*})(y,z,s) =  \Big( 1-\frac{\alpha_{g_y^{\circ}}}{\alpha_{g_z^{\circ}}} \Big) \times L_p({\bf g}, {\bf g^*})(y,z,s).
\end{equation}

Fix a point $(y,z,s) \in \mathcal W_{{\bf gh}}^{\circ} \cap \cS$. Set $L=\Q(g_y^{\circ},h_z^{\circ},\lambda_N(g_y^{\circ}))$ and let $L_p$ denote the $p$-adic completion of $L$. Define
\begin{equation}\label{log-+}
 \log^{-+}: H^1(\mathbb Q_p, \mathbb V_{g_y} \otimes V_{h_z}(1-s)) \xrightarrow{\pr^{-+}} H^1(\mathbb Q_p, \mathbb V_{g_y}^- \otimes V_{h_z}^+(1-s)) \rightarrow L_p,
\end{equation}
where the first map is the projection onto $H^1(\mathbb Q_p, \mathbb V^-_{g_y} \otimes V^+_{h_z}(1-s))$ and the last one is the composition of the Bloch--Kato logarithm with the pairing with the differential $\eta_{g^\circ_y} \otimes \omega_{h^\circ_z}$.

It follows from Theorems \ref{biglog} and  \ref{reclaw} that for all $(y,z,s) \in \mathcal W_{{\bf gh}}^{\circ} \cap \cS$:
\begin{equation}\label{exp2}
\tilde L_p({\bf g}, {\bf g^*})(y,z,s) = \Big( 1-\frac{\alpha_{g_z^{\circ}}}{p\alpha_{g_y^{\circ}}} \Big) \cdot \log^{-+} (\kappa_p({\bf g}, {\bf g^*})(y,z,s)),
\end{equation}
up to multiplication by the $c$-factor we have described in Theorem \ref{reclaw} and which does not affect to our discussion since we always work modulo $L^{\times}$.

\vskip 12pt

Observe that the function $\tilde L_p({\bf g},{\bf g^*})$ vanishes along the curve $\cC$, so the restriction of its derivative to that line is exactly zero. Recall that $\tilde L_p({\bf g},{\bf g^*})$ is a two-variable function, determined by the values of $y$ and $z$, since the third variable $\sigma$ comes automatically determined by $z$. Hence, one has that
\[ \frac{\partial}{\partial y} \tilde L_p({\bf g}, {\bf h}) + \frac{\partial}{\partial z} \tilde L_p({\bf g}, {\bf h}) = 0. \]

Recall that $\alpha_{{\bf g}}$ is an analytic function defined over $\Lambda_{{\bf g}}$; we denote by $\alpha_{{\bf g}}'$ its derivative. We may consider as before the $\mathcal L$-invariant attached to the adjoint representation of ${\bf g}$, \[ \mathcal L(\ad^0({\bf g})) = -\frac{\alpha_{{\bf g}}'}{\alpha_{{\bf g}}}. \] Observe that its specializations at classical points agree with the definitions given before in \eqref{citro}.

Now, we can compute the partial derivatives of $\tilde L_p(\hg,\hh)$ at a crystalline point $(y,y,\ell-1)$ of the curve $\cC$. Using \eqref{exp1}, one gets
\begin{equation}\label{partial-y}
\frac{\partial}{\partial y} \tilde L_p({\bf g}, {\bf h})(y,y,\ell-1) = \mathcal L(\ad^0(g_y)) \cdot L_p({\bf g},{\bf g^*})(y,y,\ell-1).
\end{equation}
Then, using \eqref{exp2}, we first observe that $\kappa_p(\hg,\hg^*)(y,z,s)$ vanishes at $(y,y,\ell-1)$ and its derivative in the $z$-direction is precisely the logarithm of the derived cohomology class we have previously computed in Proposition \ref{claseder}; hence, one gets that
\begin{equation}\label{partial-z}
\frac{\partial}{\partial z} \tilde L_p({\bf g}, {\bf h})(y,y,\ell-1) = (1-p^{-1}) \cdot (\log^{-+}(\kappa_p'({\bf g},{\bf g^*})(y,y,\ell-1))).
\end{equation}

We have then proved the following result:

\begin{theorem}\label{teobf}
For any crystalline point $(y,y,\ell-1)$ on $\cC$, it holds that \[ \mathcal L(\ad^0(g_y)) \cdot L_p({\bf g},{\bf g^*})(y,y,\ell-1) = \log^{-+} (\kappa_p'({\bf g},{\bf g^*})(y,y,\ell-1)) \pmod{L^{\times}}. \]
\end{theorem}

\subsection{Weight one modular forms}
Let $g \in S_1(N,\chi_g)$ and $h\in S_1(N,\chi_h)$ be two cuspidal eigenforms of weight one. Let $V_g$ and $V_h$ denote the Artin representations over a finite extension $L$ of $\Q$ attached to $g$ and $h$. Let $\alpha_g, \beta_g$ (resp.\,$\alpha_h, \beta_h$) denote the roots of the $p$-th Hecke polynomial of $g$ (resp.\,of $h$). We assume throughout that $\alpha_g \ne \beta_g$ and $\alpha_h \ne \beta_h$.
We also assume $L$ is large enough as specified in \S \ref{Hidafam}.

\begin{defi}
Let ${\bf g}$ and ${\bf h}$ be Hida families passing through $p$-stabilizations $g_{\alpha}$, $h_{\alpha}$ of $g$, $h$ at some point $(y_0,z_0)\in \cW_{\hg}^\circ \times \cW_{\hh}^\circ$ of weights $(1,1)$. Define \[ \kappa(g_{\alpha},h_{\alpha}):=\kappa({\bf g},{\bf h})(y_0,z_0,0) \in H^1(\mathbb Q, V_{gh}\otimes L_p(1)) \] as the  specialization of $\kappa({\bf g}, {\bf h})$ at the point $(y_0,z_0,0)$.
\end{defi}

This procedure yields four a priori different global cohomology classes:
\begin{equation}\label{cuatro}
\kappa(g_{\alpha},h_{\alpha}), \quad \kappa(g_{\alpha},h_{\beta}), \quad \kappa(g_{\beta},h_{\alpha}), \quad \kappa(g_{\beta},h_{\beta}),
\end{equation}
one for each choice of pair of roots of the $p$-th Hecke polynomials of $g$ and $h$.

Given a $p$-adic representation $V$ of $G_{\Q_p}$ with coefficients in $\Q_p$, Bloch and Kato introduced in \cite{BK} a collection of subspaces of the local Galois cohomology group $H^1(\mathbb Q_p,V)$, denoted respectively \[ 0 \subset H_e^1(\mathbb Q_p, V) \subset H_{\fin}^1(\mathbb Q_p, V) \subset H_g^1(\mathbb Q_p,V) \subset H^1(\mathbb Q_p, V). \]

\begin{defi}
Let $V$ be a representation of $G_{\Q}$ with coefficients in $\Q_p$. The group of classes that are de Rham at $p$ (i.e.\,the restriction to $\mathbb Q_p$ lies in $H_g^1(\mathbb Q_p,V)$) and unramified at all primes $\ell \neq p$ is denoted as $H_{\fin,p}^1(\mathbb Q, V)$.

The group of classes that are crystalline at $p$ and unramified at any other prime $q \neq p$ is denoted as $H_{\fin}^1(\mathbb Q, V)$.
\end{defi}

\begin{propo}
The four classes in \eqref{cuatro} lie in $H_{\fin,p}^1(\mathbb Q, V_{gh}\otimes L_p(1))$.
\end{propo}
\begin{proof}
The dimensions of $H_g^1(\Q_p,V_{gh}(1))$ and $H^1(\Q_p,V_{gh}(1))$ are equal, according to the discussion of \cite[Section 1.4]{DR3}; hence, the classes are de Rham at $p$. Furthermore, the restriction of these classes to $\mathbb Q_q$ is $0$ for $q \neq p$. In fact a stronger fact holds true: the local $\Lambda$-adic classes $\kappa_q(\hg,\hh)\in H^1(\Q_q,\mathbb{V}_{\hg \hh})$ are $0$ at all $q \neq p$. This can be argued for instance by fixing weights $(\ell,m,s)$ large enough so that for every triple $(\epsilon_1,\epsilon_2,\epsilon_3)$ of characters of arbitrary $p$-power conductor and for any point $(x,y,z)$ above $(\nu_{\ell,\epsilon_1},\nu_{m,\epsilon_2},\nu_{s,\epsilon_3})$,  $\mathbb{V}(\hg,\hh)(x,y,z)$ contains no sub-quotient isomorphic to neither $\Q_p$ nor $\Q_p(1)$. It then follows from Tate's local Euler characteristic formula (cf.\,\cite[2.5]{Nek}) that $H^1(\Q_q,\mathbb{V}_{\hg \hh}(x,y,z)) = 0$. From this it follows that  $H^1(\Q_q,\mathbb{V}_{\hg \hh})=0$ arguing as in e.g.\,\cite[Prop.\,8.1.7 or Lemma 8.2.6]{KLZ}.
\end{proof}

Recall that we let $H$ denote the Galois extension of $\Q$ cut out by $V_g \otimes V_h$.

\begin{propo}\label{Kummer}
There are natural identifications
\begin{eqnarray}\label{dims}
H_{\fin}^1(\mathbb Q, V_{gh} \otimes_L L_p(1))&=& (\mathcal O_H^{\times} \otimes_{\Z} V_{gh} \otimes_L L_p)^{G_{\Q}},  \\ \nonumber
H_{\fin,p}^1(\mathbb Q, V_{gh} \otimes_L L_p(1)) &=& (\mathcal O_H[1/p]^{\times} \otimes_{\Z} V_{gh} \otimes_L L_p)^{G_{\Q}}.
\end{eqnarray}
\end{propo}

\begin{proof}
This follows from the same arguments as in e.g.\,\cite[Prop. 2.9]{Bel}.
\end{proof}

Since $\alpha_g \neq \beta_g$, $V_g$ decomposes as a $G_{\Q_p}$-module as $V_g = V_g^{\alpha_g} \oplus V_g^{\beta_g}$, where $V_g^{\alpha_g}$ and  $V_g^{\beta_g}$ are the $G_{\Q_p}$-invariant lines on which $\Fr_p$ acts with eigenvalue $\alpha_g$ and $\beta_g$ respectively. We similarly have $V_h = V_h^{\alpha_h} \oplus V_h^{\beta_h}$, and we may define
$$
V_{gh}^{\alpha \alpha} := V_g^{\alpha_g} \otimes V_{h}^{\alpha_h}, \quad ..., \quad  V_{gh}^{\beta \beta} := V_g^{\beta_g} \otimes V_{h}^{\beta_h}.
$$
Note that these four $G_{\Q_p}$-invariant lines in $V_{gh}$  are  linearly independent even though some of the eigenvalues $\alpha_g \alpha_h, \, \alpha_g \beta_h,\, \beta_g \alpha_{h}, \, \beta_g \beta_{h}$ might be equal. Hence there is a decomposition of $G_{\Q_p}$-modules
\begin{equation}\label{gh-decom}
V_{gh} := V_g \otimes V_{h} = V_{gh}^{\alpha \alpha} \oplus ... \oplus V_{g h}^{\beta \beta}.
\end{equation}

It follows that the local class $\kappa_p(g_{\alpha},h_{\alpha})$ may be decomposed as \[ \kappa_p(g_{\alpha},h_{\alpha}) =\kappa_p^{++}(g_{\alpha},h_{\alpha}) +\kappa_p^{+-}(g_{\alpha},h_{\alpha}) +\kappa_p^{-+}(g_{\alpha},h_{\alpha}) +\kappa_p^{++}(g_{\alpha},h_{\alpha}), \] where
\begin{eqnarray}\label{deckappa}
\kappa_p^{++}(g_{\alpha},h_{\alpha}) \in H^1(\mathbb Q_p, V_{gh}^{\beta \beta}\otimes L_p(1)), &  \kappa_p^{+-}(g_{\alpha},h_{\alpha}) \in H^1(\mathbb Q_p, V_{gh}^{\beta \alpha}\otimes L_p(1)) \\ \nonumber \kappa_p^{-+}(g_{\alpha},h_{\alpha}) \in H^1(\mathbb Q_p, V_{gh}^{\alpha \beta}\otimes L_p(1)), & \kappa_p^{--}(g_{\alpha},h_{\alpha}) \in H^1(\mathbb Q_p, V_{gh}^{\alpha \alpha}\otimes L_p(1)).
\end{eqnarray}

As before, we are specially interested in the case where $h=g^* = g\otimes \chi^{-1}$, and we impose on $g$ the assumptions (H1-H2-H3) listed in the introduction; if we denote by $\{\alpha,\beta\}$ the $p$-th Hecke eigenvalues of $g$, the $p$-th Hecke eigenvalues of $h$ are $\{1/\beta,1/\alpha\}$.

Let $\hg$ and $\hh=\hg^* := \hg \otimes \chi^{-1}$ denote the Hida families over $\cW_{\hg} = \cW_{\hh}$ passing through $g_\alpha$ and $(g_\alpha)^* = (g^*)_{1/\beta}$ respectively at some point $y_0\in \cW_{\hg}$ in weight space. Thanks to our running assumptions, the main theorem of \cite{BeDi} ensures that the weight map $\cW_{\hg} \lra \cW$ is \'etale at the point associated to $g_\alpha$. It is therefore possible to fix an open subset in $\cW_{\hg}$ around $g_\alpha$ on which the weight map is an isomorphism. This way we are entitled to work under the simplifying assumptions posed in \S \ref{self-dual} and the results in loc.\,cit.\,and \S \ref{deri} may be applied.



\begin{propo}\label{global0}
The global cohomology classes $\kappa(g_{\alpha},g_{1/\beta}^*)$ and $\kappa(g_{\beta},g_{1/\alpha}^*)$ are zero.
\end{propo}
\begin{proof}
The restriction map \[ \res_p: H^1(\mathbb Q, V_{gh}\otimes L_p(1)) \rightarrow H^1(\mathbb Q_p, V_{gh}\otimes L_p(1)) \] is injective. This follows from \cite[Proposition 2.12]{Bel}, which asserts that there are natural isomorphisms
\[ H^1(\mathbb Q, V_{gh}\otimes L_p(1)) \simeq (H^{\times} \otimes V_{gh})^{G_{\mathbb Q}}, \quad H^1(\mathbb Q_p, V_{gh}\otimes L_p(1)) \simeq (H_p^{\times} \otimes V_{gh})^{G_{\mathbb Q_p}}, \] and hence the restriction map corresponds to the natural inclusion $H \hookrightarrow H_p$. From Theorem \ref{classes-zero}, it follows that $\kappa_p(g_{\alpha},g_{1/\beta}^*)$ and $\kappa_p(g_{\beta},g_{1/\alpha}^*)$ are both zero, and the result follows.
\end{proof}

Define the curve
$$
\mathcal{D} := \{(y,z,\sigma) \in \cW_{\hg \hg^*}: \, y=y_0, \quad \mathrm{w}(z)=\sigma \cdot \varepsilon_{\cyc}\} \subset \cS
$$

\begin{propo}\label{global-derived}
There exists a unique global class $\kappa_{\gamma}'(\hg,\hg^*) \in H^1(\mathbb Q, \mathbb V_{\hg \hg^*|\mathcal{D}})$ such that \[ \kappa_{\gamma}'(\hg,\hg^*)_{|\mathcal{D}} = (\gamma-1)  \cdot \kappa_{\gamma}'(\hg,\hg^*). \]
\end{propo}

\begin{proof}
The short exact sequence of $G_{\Q}$-modules  \[ 0 \rightarrow \mathbb V_{{\bf gg^*}|\mathcal{D}} \xrightarrow{\delta} \mathbb V_{{\bf gg^*}|\cD} \rightarrow V_{gh}\otimes L_p(1) \rightarrow 0 \] gives rise to the long exact sequence
\[
H^0(\mathbb Q, V_{gh}\otimes L_p(1)) \rightarrow H^1(\mathbb Q, \mathbb V_{{\bf gg^*}|\mathcal{D}}) \xrightarrow{\delta} H^1(\mathbb Q, \mathbb V_{{\bf gg^*}|\mathcal{D}}) \rightarrow H^1(\mathbb Q, V_{gh}\otimes L_p(1)). \]
Since $H^0(\Q,V_{gh}\otimes L_p(1))=0$, the vanishing of $\kappa(g_{\alpha},g_{1/\beta}^*)$ proved in Proposition \ref{global0} implies the existence of a unique element $\kappa_{\gamma}'(\hg,\hg^*) \in H^1(\mathbb Q, \mathbb V_{\hg \hg^*|\mathcal{D}})$ satisfying the claim.
\end{proof}


\begin{remark}
The {\it global} cohomology class in Proposition \ref{global-derived} only makes sense along the curve $\mathcal{D}$. Besides,  the {\it local} cohomology class constructed in Proposition \ref{claseder} exists along the whole surface $\cS$. Hence one can not define the latter on the surface $\cS$ as the restriction at $p$ of the former, although this is indeed true after restricting to $\mathcal{D}$, a fact that we shall apply right below.
\end{remark}

Set
\begin{equation}\label{global-derived-1}
\kappa'(g_{\alpha},g_{1/\beta}^*) = \frac{1}{\log_p(\gamma)} \kappa_{\gamma}'(\hg,\hg^*)(y_0,y_0,0).
\end{equation}
Since the construction of this class coincides with the one performed in the previous section once we localize at $p$ and restrict to  the curve $\mathcal{D}$, Theorem \ref{teobf} applies and we deduce that
\begin{equation}\label{rec-law-l-inv}
\mathcal L(\ad^0(g_{\alpha})) \cdot L_p({\bf g},{\bf g^*})(y_0,y_0,0) = \log^{-+}(\kappa_p'(g_{\alpha},g_{1/\beta}^*)) \pmod{L^{\times}}.
\end{equation}

Recall from \eqref{functional} that $L_p(\hg,\hg^*)$ satisfies a functional equation relating the values at $s=0$ and $s=1$ up to a simple non-zero rational constant. Together with Proposition \ref{impad} this implies that
\begin{equation}\label{rec-law-base}
L_p({\bf g},{\bf g^*})(y_0,y_0,1)^2 = \mathcal L(\ad^0(g_{y_0}))^2 = \log^{-+}(\kappa_p'(g_{\alpha},g_{1/\beta}^*)) \pmod{L^{\times}}.
\end{equation}

This formula is the key input for deriving Theorem B, the second main result of this note.

\section{Derivatives of Fourier coefficients via Galois deformation theory}\label{sec:deform}

As in previous sections, let $g \in S_1(N,\chi)$ satisfying the hypothesis of the introduction. Let \[ \varrho_g: \Gal(H_g/\Q) \hookrightarrow \GL(V_g) \simeq \GL_2(L), \quad \varrho_{\mathrm{ad}^0(g)}: \Gal(H/\Q) \hookrightarrow \GL(\mathrm{ad}^0(g)) \simeq \GL_3(L) \]
denote the Artin representations associated to $g$ and its adjoint, respectively. Here $L$ is a finite extension of $\Q$ and $H_g \supseteq H$ denote the finite Galois extensions of $\Q$ cut out by these representations. Let $\mathcal P$ denote the set of primes of $H$ lying above $p$, and fix once for all a prime $\wp \in \mathcal P$, thus determining an embedding $H \subset H_p \subset \bar{\mathbb Q}_p$ of $H$ into its completion $H_p$ at $\wp$, and an arithmetic Frobenius $\Fr_p \in \Gal(H_p/\mathbb Q_p)$.

As it occurred in \cite{DLR2}, the regularity assumptions we have imposed on $g$ imply by e.g.\,in \cite[Prop. 3.2.2]{Das} that
\[ \dim_L ( \mathcal O_H^{\times} \otimes V_{gg^*})^{G_{\mathbb Q}} = 1, \quad  \dim_L (\mathcal O_H[1/p]^{\times} \otimes V_{gg^*})^{G_{\Q}} = 3,\]
and thus
\[ \dim_L (\mathcal O_H[1/p]^{\times}/p^{\mathbb Z} \otimes V_{gg^*})^{G_{\Q}} =   \dim_L (\mathcal O_H[1/p]^{\times}/p^{\mathbb Z} \otimes \ad^0(g))^{G_{\Q}} = 2. \]

Fix two linearly independent global cohomology classes
$$
u \in H_{\fin}^1(\mathbb Q, V_{gg^*} \otimes L_p(1)), \quad v \in H_{\fin,p}^1(\mathbb Q, V_{gg^*} \otimes L_p(1))
$$
such that under the identifications provided by Proposition \ref{Kummer}, project to a basis of the two-dimensional space $(\mathcal O_H[1/p]^{\times}/p^{\mathbb Z} \otimes \ad^0(g))^{G_{\Q}}$. By a slight abuse of notation, we continue to denote
$$
u\in (\mathcal O_H^{\times}\otimes \ad^0(g))^{G_{\Q}}, \quad v \in (\mathcal O_H[1/p]^{\times}/p^{\mathbb Z} \otimes \ad^0(g))^{G_{\Q}}
$$
the resulting elements.

Recall from \eqref{gh-decom} that $V_{gh}$ admits a natural decomposition as $G_{\Q_p}$-module as the direct sum of the four different lines $V_{gg^*}^{\alpha \alpha}$, ..., $V_{gg^*}^{\beta \beta}$. Since $\ad^0(V_g)$ is the quotient of $V_{gg^*}$ by the trivial representation, \eqref{gh-decom} descends to a decomposition of $\ad^0(V_g)$ as $G_{\Q_p}$-module as
$$
\ad^0(g) = \ad^0(g)^1 \oplus \ad^0(g)^{\alpha \otimes \bar \beta} \oplus \ad^0(g)^{\beta \otimes \bar \alpha} = L\cdot e_1 \oplus L\cdot e_{\alpha \otimes \bar \beta} \oplus L \cdot e_{\beta \otimes \bar \alpha},
$$
where $\Fr_p(e_1)=e_1$, $\Fr_p(e_{\alpha \otimes \bar \beta})=\frac{\alpha}{\beta} \cdot e_{\alpha \otimes \bar \beta}$, $\Fr_p(e_{\beta \otimes \bar \alpha})=\frac{\beta}{\alpha} \cdot e_{\beta \otimes \bar \alpha}$. Note that $\alpha/\beta \ne 1$ thanks to the regularity assumption. It could be that $\alpha=-\beta$ and hence $\alpha/\beta=\beta/\alpha=-1$, but the above decomposition is still available as explained in \eqref{gh-decom}.

Restriction to the decomposition group at $p$ allows us to regard $u$ and $v$ as elements in $H^1(\mathbb Q_p, \ad^0(g) \otimes L_p(1)) = (H_p^{\times} \otimes \ad^0(g))^{G_{\Q_p}}$, and as such we may write $u$ and $v$ as
\begin{equation}\label{decom-uv}
u = u_1 \otimes e_1 +u_{\alpha \otimes \bar \beta} \otimes e_{\alpha \otimes \bar \beta} +u_{\beta \otimes \bar \alpha} \otimes e_{\beta \otimes \bar \alpha}, \, v = v_1 \otimes e_1 + v_{\alpha \otimes \bar \beta} \otimes e_{\alpha \otimes \bar \beta} +v_{\beta \otimes \bar \alpha} \otimes e_{\beta \otimes \bar \alpha},
\end{equation}
where $u_1, v_1, u_{\alpha \otimes \bar \beta}, v_{\alpha \otimes \bar \beta}, u_{\beta \otimes \bar \alpha}, v_{\beta \otimes \bar \alpha} \in H_p^{\times}$ satisfy
$$
\Fr_p(u_1)=u_1, \quad \Fr_p(u_{\alpha \otimes \bar \beta})=\frac{\beta}{\alpha} \cdot u_{\alpha \otimes \bar \beta}, \quad \Fr_p(u_{\beta \otimes \bar \alpha})=\frac{\alpha}{\beta} \cdot u_{\beta \otimes \bar \alpha}
$$
and similarly for $v$.

We can now provide the last step in our proof of Theorems A, A' and B in the introduction. As it was shown in Corollary \ref{impad}, Theorem A may be reduced to the computation of the derivative of the Fourier coefficient $a_p(\hg)$ at $y_0$, where $\hg$ stands for the unique Hida family passing through $g_{\alpha}$.

Let \[ \tilde \rho_g: G_{\mathbb Q} \rightarrow \GL_2(L_p[\varepsilon]) \] be the unique first order $\alpha$-ordinary deformation of $\rho_g$ such that \[ \det \tilde \rho_g = \chi_g(1+\log_p \chi_{\cyc} \cdot \varepsilon). \] This representation, whose existence follows from \cite{BeDi}, satisfies \[ \tilde \rho_g = (1+\varepsilon \cdot \kappa_g) \cdot \rho_g, \] for some cohomology class $\kappa_g: G_{\mathbb Q} \rightarrow \ad(\rho)$. Considering a diagonal basis for the Frobenius action (where we take the first vector to have eigenvalue $\alpha$), the matrix form of $\kappa(\sigma)$ can be expressed as \[ \kappa(\sigma) = \mat{\kappa_1(\sigma)}{\kappa_2(\sigma)}{\kappa_3(\sigma)}{\kappa_4(\sigma)}. \] We denote by $\tilde \rho_{g,p}$ the restriction of $\tilde \rho_g$ to the decomposition group at $p$; in the same way, the restriction of $\kappa$ to the decomposition group at $p$ is denoted by $\kappa_p$, and similarly we denote by $\kappa_{i,p}$ the restriction of $\kappa_i$ to $G_{\mathbb Q_p}$.

In \cite[Lemmas 2.3 and 2.5]{BeDi} the authors determine the tangent space to a deformation problem which can be seen to be equivalent to ours. They conclude that there is a natural bijection between this tangent space and a certain subspace of $H^1(\mathbb Q, \ad(\rho_g))$; in this case, it consists on those classes $\kappa$ whose matrix representation satisfies
\begin{equation}
\kappa_{3,p}(\sigma)=0, \qquad \kappa_{1,p}(\sigma)|_{I_p}=0,
\end{equation}
being $I_p$ the inertia group at $p$. In \cite{DLR4}, the restriction of $\kappa$ to the inertia group at $p$ was determined, and the identifications of class field theory allow us to extend this to the whole decomposition group. A similar setting, where also a quite related deformation problem arises, is exploded in \cite{BDP} to treat the case of weight one Eisenstein points.

Let $V_g^{\alpha}$ be the \'etale subspace on which the action of the Frobenius is unramified for the $\Lambda$-adic representation. Restricting to the decomposition group at $p$, we have that \[ \tilde \rho_{g,p}|_{V_g^{\alpha}} = (1+\varepsilon \cdot \kappa_{1,p}) \cdot \alpha_g, \] where $\kappa_{1,p} \in H^1(\mathbb Q_p, \Hom(V_g^{\alpha},V_g^{\alpha}))$. If $g_{\alpha}'$ stands for the derivative of $\hg$ evaluated at $y_0$, we then have \[ a_p(g_{\alpha})+\varepsilon \cdot a_p(g_{\alpha}') =\alpha_g + \varepsilon \cdot \kappa_{1,p}(\Fr_p) \cdot \alpha_g, \] and consequently we have the following.

\begin{propo}
Let $\hg$ be the Hida family through $g_{\alpha}$. Then, it holds that \[ a_p(g_{\alpha}') = \kappa_{1,p}(\Fr_p) \pmod{L^{\times}}. \]
\end{propo}

Taking into account the identifications provided by class field theory, one can make $\kappa_1(\Fr_p)$ explicit.

From \cite[Section 3.2]{BeDi}, there is an exact sequence \[ 0 \rightarrow \Hom(G_H, \bar{\mathbb Q}_p) \rightarrow \Hom((\mathcal O_H \otimes \mathbb Q_p)^{\times}, \bar{\mathbb Q}_p) \rightarrow \Hom(\mathcal O_H^{\times} \otimes \mathbb Q_p,\bar{\mathbb Q}_p). \] Similarly, one has another exact sequence  \[ 0 \rightarrow \Hom(G_H, \bar{\mathbb Q}_p) \rightarrow \Hom((H \otimes \mathbb Q_p)^{\times}, \bar{\mathbb Q}_p) \rightarrow \Hom(\mathcal O_H[1/p]^{\times} \otimes \mathbb Q_p,\bar{\mathbb Q}_p). \]

Consequently, we have identifications \[ H^1(\mathbb Q, \bar{\mathbb Q}_p) \otimes \ad(\rho) \simeq (H^1(H,\bar{\mathbb Q}_p) \otimes \ad(\rho))^{G_{\mathbb Q}} \simeq (\Hom(G_H, \bar{\mathbb Q}_p) \otimes \ad(\rho))^{G_{\mathbb Q}}, \] and this corresponds with the subspace of homomorphisms of \[  (\Hom((H \otimes \mathbb Q_p)^{\times}, \bar{\mathbb Q}_p) \otimes \ad(\rho))^{G_{\mathbb Q}} \] vanishing at $\mathcal O_H[1/p]^{\times} \otimes \mathbb Q_p$.

We now recall some results which allow us to determine each of the $\kappa_{i,p}$. Since there are isomorphisms \[ (\mathcal O_H \otimes \mathbb Q_p)^{\times} \simeq \prod_{\mathfrak q \in \mathcal P} \mathcal O_{H_{\mathfrak q}}^{\times}, \qquad (H \otimes \mathbb Q_p)^{\times} \simeq \prod_{\mathfrak q \in \mathcal P} H_{\mathfrak q}^{\times}, \] we may write elements in $(\mathcal O_H \otimes \mathbb Q_p)^{\times}$ (or $(H \otimes \mathbb Q_p)^{\times}$) as tuples $(x_i)_{i \in \mathcal P}$. The action of the Galois group $G=\Gal(H/\mathbb Q)$ is transitive on $\mathcal P$, so any Galois equivariant homomorphism from $(H \otimes \mathbb Q_p)^{\times}$ (resp. ($\mathcal O_H \otimes \mathbb Q_p)^{\times}$)  to $\bar{\mathbb Q}_p$ is completely determined by its values on $H_p^{\times}$ (resp. $\mathcal O_{H_p}^{\times}$).

From class field theory, one has two distinguished elements in $H^1(\mathbb Q_p, \bar{\mathbb Q}_p)$:
\begin{enumerate}
\item The class $\kappa_{\nr}$, which is the unique homomorphism \[ \kappa_{\nr} \in \Hom(\Gal(\mathbb Q_p^{\nr}/\mathbb Q_p), \bar{\mathbb Q}_p) \] taking $\Frob_p$ to $1$.
\item The restriction to $G_{\mathbb Q_p}$ of the logarithm of the cyclotomic character \[ \kappa_{\cyc} := \log_p(\varepsilon_{\cyc}), \] which gives a ramified element of $H^1(\mathbb Q_p, \bar{\mathbb Q}_p)$.
\end{enumerate}

Furthermore, observe that $H^1(\mathbb Q_p, \bar{\mathbb Q}_p)$ is identified with $\Hom(G_{H_p}, \bar{\mathbb Q}_p)^{\Gal(H_p/\mathbb Q_p)}$ as explained above. In the latter space, we denote by $\kappa_{\nr}$ the morphism that takes $\Fr_p$ to $1$, and by $\kappa_{\cyc}$ the ramified element defined by $\log_p(\varepsilon_{\cyc})$.

\vskip 12pt

Now we can determine explicitly the element $\kappa \in H^1(\mathbb Q, \ad(\rho))$, by constructing a homomorphism \[ \Phi_g(x): (H \otimes \mathbb Q_p)^{\times} \rightarrow H_p \otimes \ad(\rho) \] corresponding to $\kappa$ via the previous identifications. In particular, it vanishes when evaluated at the basis $\{u,v\}$ of units for the adjoint. Moreover, $\Phi_g(\pi_p^{-1})$ leaves invariant the one-dimensional space $V_g^{\alpha}$, being $\pi_p^{-1}$ the id\`ele which is equal to the inverse of a local uniformiser of $H_p$ for the fixed prime $\wp$ above p, and to $1$ everywhere else. The eigenvalue for the action of $\Phi_g(\pi_p^{-1})$ on that subspace is precisely $\kappa_1(\Fr_p)$.

\vskip 12pt

Let $u_g^{\times}$ be any generator of $(\mathcal O_H^{\times} \otimes \ad(\rho))^{G_{\mathbb Q}}$, and let $v_g^{\times}$ be any element of the space $(\mathcal O_H[1/p]^{\times} \otimes \ad(\rho))^{G_{\mathbb Q}}$ such that $\{u_g^{\times}, v_g^{\times}\}$ is a basis of  $(\mathcal O_H[1/p]^{\times}/p^{\mathbb Z} \otimes \ad(\rho))^{G_{\mathbb Q}}$. Let \[ u_g := (\log_p \otimes \id)(u_g^{\times}), \quad v_g := (\log_p \otimes \id)(v_g^{\times}), \quad \tilde v_g := (\ord_p \otimes \id)(v_g^{\times}) \quad \in H_p \otimes \ad(\rho). \]

Consider the element \[ A_g \in (H_p \otimes \ad(\rho))^{G_{\mathbb Q_p}}, \] which on an eigenbasis for the Frobenius takes the form \[ A_g: \mat{0}{\frac{\log_p(u_1)}{\log_p(u_{\alpha \otimes \bar \beta})}}{0}{1}. \] Consider also $J \in (H_p \otimes \ad(\rho))^{G_{\mathbb Q_p}}$, which in the same basis is expressed as \[ J = \mat{a}{b}{0}{c}, \] where $a,b,c \in H_p$.

These choices give rise to the $G_{\mathbb Q_p}$-equivariant homomorphism $\Phi_g: H^{\times} \rightarrow H_p \otimes \ad(\rho)$ given by the rule
\begin{equation}
\Phi_g(x) := \sum_{\sigma \in G} \log_p(^{\sigma} x) \cdot (\sigma^{-1} \cdot A_g) + \sum_{\sigma \in G} \ord_p(^{\sigma} x) \cdot (\sigma^{-1} \cdot J).
\end{equation}

Here, $\sigma^{-1} \cdot A_g$ denotes the action of $\sigma^{-1}$ by conjugation on the second factor of the tensor product, $\ad(\rho)$; then, this homomorphism can be extended to the whole $(H \otimes \mathbb Q_p)^{\times}$. The aim is to determine a suitable $J$ such that $\Phi_g(x)$ corresponds to $\kappa$ via class field theory. The following result explains the behavior of the terms coming from the {\it $\log$-part}.

\begin{propo}
Let $\ad(\rho)^{\ord}:=\Hom(V_g/V_g^{\alpha},V_g)$. The homomorphism $\Phi_g$ vanishes on $\mathcal O_H^{\times} \otimes \mathbb Q_p$, and $\Phi_g(x) \subset H_p \otimes \ad(\rho)^{\ord}$ for any $x$ of the form $(x_p,1, \ldots, 1)$, where $x_p \in \mathcal O_{H_p}^{\times}$. Moreover, $\Phi_g(x)$ fixes $V_g^{\alpha}$ for all $x$ of the form $(x_p,1,\ldots,1)$, with $x_p \in H_p^{\times}$.
\end{propo}
\begin{proof}
This follows from \cite[Lemma 1.6]{DLR4}, where $\Phi_g$ is defined in the same way but without the $\ord$-terms. It is clear that the behavior at the $\mathcal O_H$-units is not affected by the presence of these extra terms. The last part of the statement follows from the definition of $A_g$ and $J$.
\end{proof}

To determine the remaining parameters of the matrix $J$ ($a$, $b$ and $c$), we first observe that the homomorphism $\Phi_g$ must vanish at $\mathcal O_H[1/p]^{\times} \otimes \mathbb Q_p$. The endomorphism $v_g$ is represented by a matrix of the form \[ v_g: \mat{\log_p(v_1)}{\log_p(v_{\beta \otimes \bar \alpha})} {\log_p(v_{\alpha \otimes \bar \beta})}{-\log_p(v_1)}, \] and hence
\begin{equation}\label{traza}
\Tr(A_g v_g) = \frac{\log_p(u_1) \cdot \log_p(v_{\alpha \otimes \bar \beta})-\log_p(u_{\alpha \otimes \bar \beta}) \cdot \log_p(v_1)}{\log_p(u_{\alpha \otimes \bar \beta})}.
\end{equation}

Furthermore, observe that the restrictions of both $\kappa_{1,p}$ and $\kappa_{4,p}$ to the decomposition group at $p$ belong to \[  H^1(\mathbb Q_p, \bar{\mathbb Q}_p) \simeq H^1(H_p, \bar{\mathbb Q}_p)^{G_{\mathbb Q_p}}; \] this space is two-dimensional and it is generated by $\kappa_{\cyc}$ and $\kappa_{\nr}$; in the same way, $\kappa_{2,p}$ is a cohomology class in the one-dimensional space \[  H^1(H_p, \bar{\mathbb Q}_p(\beta/\alpha))^{G_{\mathbb Q_p}}. \]

It is clear that \[ \kappa_1|_{I_p} = 0, \quad \kappa_2|_{I_p} = \frac{\log_p(u_1)}{\log_p(u_{\alpha \otimes \bar \beta})} \kappa_{\cyc}, \qquad \kappa_4|_{I_p} = \kappa_{\cyc}. \]

Then, to determine the restriction of the cohomology classes to the whole decomposition group, we impose these three conditions:
\begin{enumerate}
\item The fact that the trace of the representation is prescribed (and is equal to $\kappa_{\cyc}$) forces that \[ \kappa_1 = \lambda \cdot \kappa_{\nr}, \qquad \kappa_4 = \kappa_{\cyc}-\lambda \cdot \kappa_{\nr}. \]
\item The fact that $H^1(H_p, \bar{\mathbb Q}_p(\beta/\alpha))^{G_{\mathbb Q_p}}$ is one-dimensional (since $\alpha \neq \beta$), makes that \[ \kappa_2 = \frac{\log_p(u_1)}{\log_p(u_{\alpha \otimes \bar \beta})} \kappa_{\cyc} \] (i.e., there is no contribution coming from the Frobenius).
\item The remaining parameter $\lambda$ can be determined by considering the associated matrix $J$ giving rise to the endomorphism $\Phi_g(x)$. This is the content of the following proposition.
\end{enumerate}

\begin{propo}
There exists a unique $\lambda$ such that the homomorphism $\Phi_g(x)$ vanishes at $\mathcal O_H[1/p]^{\times} \otimes \mathbb Q_p$.
\end{propo}
\begin{proof}
Picking $w \in \mathcal O_H[1/p]^{\times}$ and an arbitrary $B \in \ad(\rho)$, it is enough to see that \[ \Tr(\Phi_g(w) \cdot B) = 0, \] due to the non-degeneracy of the $H_p$-valued trace pairing on $H_{\wp} \otimes \ad(\rho)$. Set \[ w_g^{\times} := \sum_{\sigma \in G} \, ^{\sigma} w \otimes (\sigma \cdot B) \in (\mathcal O_H^{\times} \otimes \ad(\rho))^{G_{\mathbb Q}}. \] Observe that $w_g^{\times}$ can be expressed in terms of the basis $u_g^{\times}$ and $v_g^{\times}$. In particular, let \[ w_g := (\log_p \otimes \id)(w_g^{\times}) = \lambda \cdot u_g + \mu \cdot v_g; \quad \tilde w_g := (\ord_p \otimes \id)(w_g^{\times}) = \mu \cdot \tilde v_g. \]

Then, \[ \Tr(\Phi_g(w) \cdot B) = \lambda \cdot \Tr( A_g \cdot u_g) + \mu \cdot \Tr(A_g v_g-J \cdot \tilde v_g). \] Observe that \[ J = \mat{-\lambda}{0}{0}{\lambda}, \] and from \eqref{traza}, one sees that \[ \lambda = \frac{\log_p(u_1) \cdot \log_p(v_{\alpha \otimes \bar \beta})-\log_p(v_1) \cdot \log_p(u_{\alpha \otimes \bar \beta})}{2 \cdot \log_p(u_{\alpha \otimes \bar \beta}) \cdot \ord_p(v_1)}. \]
\end{proof}

Hence, we conclude that \[ \kappa_1 =  -\frac{\log_p(u_1) \cdot \log_p(v_{\alpha \otimes \bar \beta})-\log_p(v_1) \cdot \log_p(u_{\alpha \otimes \bar \beta})}{2 \cdot \log_p(u_{\alpha \otimes \bar \beta}) \cdot \ord_p(v_1)} \cdot \kappa_{\nr}, \] and evaluating at $\Fr_p$ we finally obtain the formula we anticipated below
\begin{equation}\label{asclaimed}
a_p(g_{\alpha}') = \frac{\log_p(u_1) \cdot \log_p(v_{\alpha \otimes \bar \beta})-\log_p(v_1) \cdot \log_p(v_{\alpha \otimes \bar \beta})}{\log_p(u_{\alpha \otimes \bar \beta})} \pmod{L^{\times}},
\end{equation}
as claimed.

\section{Proof of main results}\label{proof-main}

\subsection{Proof of Theorems A and A'}
Recall that by \cite[Proposition 4.2]{DLR2}, the $p$-adic iterated integral of the statement agrees, up to multiplication by a scalar in $L^{\times}$, with the special value $L_p(\hg,\hg^*)(y_0,y_0,1)$.
Further, applying the relation between $L_p({\bf g},{\bf g}^*)(y_0,y_0,1)$ and $a_p(g_{\alpha}')$ as described in Corollary \ref{impad}, we have that \[ L_p({\bf g},{\bf g^*})(y_0,y_0,1) = L_p({\bf g},{\bf g^*})(y_0,y_0,0) = \mathcal L(\ad^0(g_\alpha)) = a_p(g_{\alpha}') \pmod{L^{\times}}. \] The derivative of the Fourier coefficient was computed in the previous section, and it follows from \eqref{asclaimed} that \begin{equation}\label{Linv} L_p({\bf g},{\bf g^*})(y_0,y_0,1) = \frac{\log_p(u_1) \cdot \log_p(v_{\alpha \otimes \bar \beta})-\log_p(v_1) \cdot \log_p(u_{\alpha \otimes \bar \beta})}{\log_p(u_{\alpha \otimes \bar \beta})} \pmod{L^{\times}}. \end{equation} This proves Theorems A and A'.

\subsection{Proof of Theorem B}\label{proof-b}
We may combine \eqref{rec-law-base} with the above result to deduce that \[ \log^{-+} (\kappa_p'(g_{\alpha},g_{1/\beta}^*)) =  \Big( \, \frac{ \log_p(u_1)\log_p(v_{\alpha \otimes \bar \beta}) - \log_p(u_{\alpha \otimes \bar \beta})\log_p(v_1)}{\log_p(u_{\alpha \otimes \bar \beta})} \, \Big)^2  \pmod{L^{\times}}. \]

Recall that in \eqref{log-+} we defined the map $\log^{-+}$ as the composition of the Perrin-Riou big logarithm of Proposition \ref{biglog} specialized at weight $(y_0,y_0,0)$ and the pairing with the class $\eta_{g_{\alpha}} \otimes \omega_{g_{1/\beta}^*}$ introduced in \S \ref{Hidafam}. These differential classes satisfy \[ \langle \eta_{g_{\alpha}}, \omega_{g_{1/\beta}^*}\rangle = \frac{1}{\lambda_N(g_{\alpha}) \mathcal E_0(g_{\alpha}) \mathcal E_1(g_{\alpha})} \in L^{\times}. \] under the perfect pairing $D(V_g) \times D(V_{g^*}) \rightarrow L_p$. We may take a decomposition of $V_g$ and $V_{g^*}$ as $G_{\mathbb Q_p}$-modules \[ V_g = L \cdot e_{\alpha}^g \oplus L \cdot e_{\beta}^g, \qquad V_{g^*}= L \cdot e_{1/\alpha}^{g^*} \oplus L \cdot e_{1/\beta}^{g^*}, \] respectively, where $\{e_{\alpha}^g, e_{\beta}^g\}$ and $\{e_{1/\alpha}^{g^*},e_{1/\beta}^{g^*}\}$ are basis of $V_g$ and $V_{g^*}$, one dual of each other, and compatible with the choice of the basis for the tensor product $V_{gg^*}$ considered at the previous section. As explained in \cite[\S 2]{DR2.5}, one may define $p$-adic periods
\begin{equation}
\Xi_{g_{\alpha}} \in H_p^{\Fr_p = \beta^{-1}}, \qquad \Omega_{g_{1/\beta}^*} \in H_p^{\Fr_p = \beta}
\end{equation}
satisfying that \[ \Xi_{g_{\alpha}} \otimes e_{\beta}^g = \eta_{g_{\alpha}}, \qquad \Omega_{g_{1/\beta}^*} \otimes e_{1/\beta}^{g^*} = \omega_{g_{1/\beta}^*}. \]  The natural pairing between $D(V_g) = (H_p \otimes V_g)^{G_{\mathbb Q_p}}$ and  $D(V_{g^*}) = (H_p \otimes V_{g^*})^{G_{\mathbb Q_p}}$ induces a duality between $V_g$ and $V_{g^*}$, and hence the quantity \[ \langle \eta_{g_{\alpha}}, \omega_{g_{1/\beta}^*} \rangle = \Xi_{g_{\alpha}} \cdot \Omega_{g_{1/\beta}^*} \cdot  \langle e_{\beta}^g, e_{1/\beta}^{g^*} \rangle =   \Xi_{g_{\alpha}} \cdot \Omega_{g_{1/\beta}^*} \] belongs to $L^{\times}$. Consequently, the class
\[ \kappa_\circ = \frac{ \log_p(u_1)\log_p(v_{\alpha \otimes \bar \beta}) - \log_p(u_{\alpha \otimes \bar \beta})\log_p(v_1)}{\log_p(u_{\alpha \otimes \bar \beta})^2} \times  \Big(  \log_p(v_{\alpha \otimes \bar \beta}) u - \log_p(u_{\alpha \otimes \bar \beta})v \Big) \] satisfies $\log^{-+}(\kappa_p'(g_{\alpha},g_{1/\beta}^*))=\log^{-+}(\res_p(\kappa_\circ))$.

We may write the cohomology class $\kappa'(g_{\alpha},g_{1/\beta}^*)$ as a linear combination \[ \kappa'(g_{\alpha},g_{1/\beta}^*) = a \cdot u + b \cdot v + c \cdot p, \] where $a,b,c \in L_p$.

The condition for an element to lie in the kernel of the map $\log^{-+}$ is \[ a \cdot \log_p(u_{\alpha \otimes \bar \beta}) + b \cdot \log_p(v_{\alpha \otimes \bar \beta}) = 0. \] Hence, we have that \[ \kappa'(g_{\alpha},g_{1/\beta}^*) = \kappa_\circ + \lambda(\log_p(v_{\alpha \otimes \bar \beta}) \cdot u - \log_p(u_{\alpha \otimes \bar \beta}) \cdot v) + \mu \cdot p. \]

Consider now the map \[ \log^{--}: H^1(\mathbb Q_p, V_{gh}(1)) \xrightarrow{\pr^{--}} H^1(\mathbb Q_p, V_{gh}^{--}(1)) \simeq H^1(\mathbb Q_p, \Q_p(\alpha/\beta)(1)) \xrightarrow{\log_{\BK}} L_p. \]

According to \eqref{menos}, the class $\kappa_p'(g_{\alpha},g_{1/\beta}^*)$ also lies in the kernel of $\log^{--}$. Hence, taking equalities up to periods, \[ \log^{--}(\kappa_p'(g_{\alpha},g_{1/\beta}^*)) = \log^{--}(\res_p(\kappa_\circ)) + \lambda(\log_p(v_{\alpha \otimes \bar \beta}) \cdot \log_p(u_1) - \log_p(u_{\alpha \otimes \bar \beta}) \cdot \log_p(v_1)) \] \[ = \lambda \cdot (\log_p(v_{\alpha \otimes \bar \beta}) \cdot \log_p(u_1) - \log_p(u_{\alpha \otimes \bar \beta}) \cdot \log_p(v_1)) = \lambda \cdot \log_p(u_{\alpha \otimes \bar \beta}) \cdot L_p({\bf g},{\bf g^*})(y_0,y_0,1) = 0. \]

The assumption $\mathcal L(\ad^0(g_{\alpha})) \neq 0$ in Theorem B implies  that $L_p({\bf g},{\bf g^*})(y_0,y_0,1) \neq 0$ by the first display in this section. Hence $\lambda=0$ and Theorem B follows.

\subsection{Non-vanishing of the $\cL$-invariant}\label{sec:L}

We conclude this article by noting that the non-vanishing of the $\cL$-invariant can be proved in most dihedral cases, because the expression \eqref{Linv} simplifies considerably. Indeed, let $K$ be a real or imaginary quadratic field of discriminant $D$ and let $\psi: G_K \lra L^\times$ be a finite order character of conductor $\mathfrak c\subset \mathcal O_K$ (and of mixed signature at the two archimedean places if $K$ is real). Then the theta series $g=\theta(\psi)$ attached to $\psi$ is an eigenform of weight $1$, level $N_g = |D|\cdot \mathbf{N}_{K/\Q}(\mathfrak c)$ and nebentype $\chi_g=\chi_K \chi_\psi$, where $\chi_K$ is the quadratic character associated to $K/\Q$ and $\chi_\psi$ is the central character of $\psi$.

Let $\psi'$ denote the $\Gal(K/\Q)$-conjugate of $\psi$ defined by the rule $\psi'(\sigma) = \psi(\sigma_0 \sigma \sigma_0^{-1})$ for any choice of $\sigma_0\in \G_{\Q}\setminus G_K$. If $\psi\ne \psi'$ then $g$ is cuspidal.

Let $p\nmid N_g$ be a prime number and fix an embedding $\bar\Q \subset \bar\Q_p$. In line with the introduction, suppose that hypotheses (H1-H2-H3) are fulfilled. This amounts to asking that
\begin{enumerate}

\item[(i)] $\psi \ne \psi' \, (\mathrm{mod} \, p)$, so that $g$ is cuspidal even residually at $p$.

\item[(ii)] If $K$ is imaginary and $p=\wp \bar\wp$ splits in $K$, then $\psi(\wp)\ne \psi'(\wp) \, (\mathrm{mod} \, p)$.

\item[(iii)] If $K$ is real, $p$ does not split. If $K$ is imaginary and the field $H=\bar K^{\mathrm{ker}(\psi)}$ cut out by $\psi$ has Galois group $\Gal(H/\Q)=D_4$, the dihedral group of order $8$, then $p$ does not split in the single real quadratic  field contained in $H$.
\end{enumerate}

\begin{proposition}\label{Linv-simple} If $K$ is imaginary with $p$ split, or $K$ is real with $p$ inert, then
\begin{equation}
\mathcal L(\ad^0(g_\alpha)) = \log_p(v_1) \quad  \pmod{L^{\times}}.
\end{equation}
In particular $\mathcal L(\ad^0(g_\alpha)) \ne 0$.
\end{proposition}

\begin{proof}
If $K$ is real (and thus $p$ remains inert in it) then $u_1$ is the norm of the fundamental unit of $K$ and hence its $p$-adic logarithm vanishes. If $K$ is imaginary and $p$ splits, then $v$ is a $p$-unit in $K^\times$ and hence $\log_p(v_{\alpha \otimes \bar \beta})=0$.
\end{proof}

Note that Theorem B simplifies considerably in the setting of the above proposition. The simplest scenario where it is not a priori obvious that the $\cL$-invariant is non-zero arises when $K$ is imaginary but $p$ remains inert. Fix a prime $\wp$ in $H$ above $p$ and set $\psi_{\mathrm{ad}} = \psi/\psi'$. Note that  $\psi_{\mathrm{ad}}$ is a ring class character, regardless of whether $\psi$ is so or not. Let $u_{\psi_{\mathrm{ad}}}$ (resp.\,$v_{\psi_{\mathrm{ad}}}$) denote any element spanning the $1$-dimensional $L$-vector space
$$
\cO_H^\times[\psi_{\mathrm{ad}}] := \{ x\in \cO_H^\times \otimes L: \sigma(x)=\psi_{\mathrm{ad}}(\sigma)x, \quad \sigma \in \Gal(H/K) \},
$$
respectively $\frac{\cO_H[1/\wp]^\times[\psi_{\mathrm{ad}}]}{\cO_H^\times[\psi_{\mathrm{ad}}]}$. Set also $u'_{\psi_{\mathrm{ad}}} = \Frob_{\wp} u_{\psi_{\mathrm{ad}}}$ and $v'_{\psi_{\mathrm{ad}}} = \Frob_{\wp} v_{\psi_{\mathrm{ad}}}$. Then it is a straight-forward computation to check that the $\cL$-invariant appearing in Theorems A and A' is
\begin{equation}
\mathcal L(\ad^0(g_\alpha)) = \frac{\log_{\wp}(u_{\psi_{\mathrm{ad}}})\log_{\wp}(v'_{\psi_{\mathrm{ad}}}) - \log_{\wp}(u'_{\psi_{\mathrm{ad}}})\log_{\wp}(v_{\psi_{\mathrm{ad}}})}{\log_{\wp}(u_{\psi_{\mathrm{ad}}}) - \log_{\wp}(u'_{\psi_{\mathrm{ad}}})}  \quad  \pmod{L^{\times}}.
\end{equation}

\section{Darmon--Dasgupta units and factorization of $p$-adic $L$-functions}\label{sec:DD}

\subsection{Darmon--Dasgupta units and Gross' conjecture}

Let us place ourselves again in the setting of \S \ref{sec:L}, where $K$ is real and $p$ remains inert in it. In this scenario Darmon and Dasgupta \cite{DD} associated to the ring class character $\psi_{\mathrm{ad}}$ a local unit $v_{\Dd}[\psi_{\mathrm{ad}}] \in K_p^\times$ and conjectured that $v_{\Dd}[\psi_{\mathrm{ad}}]$ actually belongs to $\cO[1/\wp]^\times[\psi_{\mathrm{ad}}]$. The combination of \cite[Theorem 4.4]{Park} and Darmon--Dasgupta-Pollack's \cite[Theorem 2]{DDP}  provides strong evidence for this conjecture, as putting  these results together it follows that, in our notations,
\begin{equation}\label{DDPark}
\log_{p}(\mathbf{N}_{K_p/\Q_p} (v_{\Dd}[\psi_{\mathrm{ad}}])) = \log_p(v_1) \quad  \pmod{L^{\times}}.
\end{equation}

The above equality together with Theorems B and C yield a formula relating the derived Beilinson--Flach elements of this note with Darmon--Dasgupta units, much in the spirit of \cite[Theorem A]{BSV} and \cite[Theorem C]{DR3} for diagonal cycles versus Stark-Heegner points. Taking into account the decomposition introduced in \eqref{decom-uv}, the element $\kappa'(g_{\alpha},g_{1/\beta}^*)_1$ belongs to $\mathbb Q_p^{\times} \otimes L_p$.

\begin{coro}\label{coro-DD} Let $g=\theta(\psi)$ be the theta series associated to a finite order character $\psi$ of mixed signature of a real quadratic field $K$. Let $p$ be a prime that remains inert in $K$. Then, \[ \kappa'(g_\alpha,g^*_{1/\beta})_1 = \log_{p}(\mathbf{N}_{K_p/\Q_p} (v_{\Dd}[\psi_{\mathrm{ad}}])) \cdot \mathbf{N}_{K_p/\Q_p} (v_{\Dd}[\psi_{\mathrm{ad}}]) \pmod{L^{\times}}. \]
\end{coro}

In spite of the ostensible parallelism between the above formula and \cite[Theorem A]{BSV} and \cite[Theorem C]{DR3}, note that the proof of Corollary \ref{coro-DD} follows quite a different route from \cite{BSV} and \cite{DR3}. The main reason is that in the latter two references it was crucially exploited a factorization of $p$-adic $L$-functions, which follows  from a comparison of critical values.

In our setting here one still expects to have an analogous factorization, but proving it appears to be far less trivial. Since this issue poses intriguing questions, and our results shed some light on them, we discuss it in more detail below.

\subsection{Artin $p$-adic $L$-functions}

Let $g=\theta(\psi)\in S_1(N,\chi)$ be a theta series of a quadratic field $K$ and $p$ be a prime which is split (resp.\,inert) if $K$ is imaginary (resp.\,real). Keep  the assumptions of Proposition \ref{Linv-simple}, and set  $h=g^*=\theta(\psi^{-1})$ and $\psi_{\mathrm{ad}} = \psi/\psi'$ as usual.

Let
$L_p(g_\alpha,g^*_{1/\beta},s)$ denote the cyclotomic $p$-adic Rankin-Hida $L$-function associated to the pair $(g_\alpha,g^*_{1/\beta})$ of $p$-stabilizations of $g$ and $h$, as in \S \ref{HidaRankin}. Note that this $p$-adic $L$-function has no critical points.
 Nevertheless, since
\begin{equation}\label{rep-decom}
V_{g}\otimes V_{g^*} \simeq 1 \oplus \chi_K \oplus \mathrm{Ind}^K_{\Q}(\psi_{\mathrm{ad}}),
\end{equation}
one may still wonder whether $L_p(g_\alpha,g^*_{1/\beta},s)$ admits a factorization mirroring the one satisfied by its classical counterpart:
\begin{equation}\label{L-decom}
L(g,h,s) = \zeta(s) \cdot L(\chi_K,s) \cdot L(K,\psi_{\mathrm{ad}},s).
\end{equation}
While \eqref{L-decom} follows directly from \eqref{rep-decom} by Artin formalism, a putative analogous factorization of $L_p(g_\alpha,g^*_{1/\beta},s)$ is far less trivial. In the CM case one can prove it by a method which is nowadays standard, but rather deep as one needs firstly to extend $L_p(g_\alpha,g^*_{1/\beta},s)$ to a two-variable $p$-adic $L$-function, prove a factorization in this scenario, and then invoke Gross's theorem \cite{Gr}. Since we did not find it in the literature and the output is not precisely what one would na\"ively expect\footnote{Indeed, the formula $L_p(g,g^*,s) = \mathfrak f(s) \cdot \zeta_p(s) \cdot L_p(\chi_K \omega, s) \cdot L_p^{\Katz}(\psi_{\ad} \cdot \mathbb N^{s})$ is not correct, in spite of being the direct analogue of \eqref{L-decom}. For one thing, this formula would enter in contradiction with Theorem A.}, we provide it below.

As a piece of notation, we say that an element $\mathfrak f$ in a finite algebra over $\Lambda^{\otimes n}$ for some $n\geq 0$, is an {\em $L$-rational fudge factor} if it is a rational function with coefficients in $L$  which extends to an Iwasawa function with neither poles nor zeroes at crystalline classical points.  Let also $L_p^{\Katz}$ denote Katz's $p$-adic $L$-function on the space of  Hecke characters of an imaginary quadratic field. If $\xi$ is one such a character of $K$, we write $ L_p^{\Katz}(\xi,s) =  L_p^{\Katz}(\xi \cdot \mathbb N^{s})$ where  $\mathbb N$ stands for the Hecke character of infinity type $(1,1)$ induced by the norm from $K$ to $\mathbb Q$.




\begin{theorem}\label{factorizationCM-in-zetas} Assume $K$ is imaginary and $p$ splits in it.
Then there exists an $L$-rational fudge factor $\mathfrak f\in \Lambda$ such that
\begin{equation*}
L_p(g,g^*,s) = \frac{\mathfrak f(s)}{\log_p(u_{\psi_{\ad}})} \cdot \zeta_p(s) \cdot L_p(\chi_K \omega, s) \cdot L_p^{\Katz}(\psi_{\ad},s).
\end{equation*}
\end{theorem}

\begin{proof} We follow the notations and normalizations adopted in \cite[\S 3]{DLR1} and \cite[\S 4]{DLR2}. Fix a prime $\wp$ of $K$ above $p$. Take a Hecke character $\lambda$ with image in $\mathbb Z_p^{\times}$ of infinity type $(0,1)$ and conductor $\bar{\wp}$. For every integer $\ell \geq 1$ define $\psi_{g,\ell-1}^{(p)} = \psi_g \langle \lambda \rangle^{\ell-1}$ and let $\psi_{g,\ell-1}$ be the Hecke character given by \[ \psi_{g,\ell-1}(\mathfrak q) = \begin{cases} \psi_{g,\ell-1}^{(p)}(\mathfrak q) & \text{ if } \mathfrak q \neq \bar{\wp}, \\ \chi(p)p^{\ell-1}/\psi_{g,\ell-1}^{(p)}(\wp) & \text{ if } \mathfrak q = \bar{\wp}. \end{cases} \]
As explained in loc.\,cit.\,there is a $p$-adic family $\psi_{\hg}$ of Hecke characters whose weight $\ell$ specialization is $\psi^{(p)}_{g,\ell-1}$ and such that the Hida family $\hg$ passing through $g_\alpha$ satisfies $g_\ell^\circ=\theta(\psi_{g,\ell-1})$.

Given a pair of classical weights $(\ell,s)$, define the Hecke characters \[ \Psi_{gh}(\ell,s) = \psi_{g,\ell-1}^{-1} \cdot \psi_g' \cdot \mathbb N^{s}, \qquad \Psi_{gh'}(\ell,s) = \psi_{g,\ell-1}^{-1} \cdot \psi_g \cdot \mathbb N^{s}, \qquad \Psi_g(\ell) = \psi_{g,\ell-1}^{-2} \chi \mathbb N^{\ell}. \]

All pairs $(\ell,s)$ such that $\ell > s \geq 1$ belong to the region of interpolation of both Rankin-Hida's $p$-adic $L$-function $L_p(\hg,g^*_{1/\beta},s)$ and Katz's $p$-adic $L$-functions $L_p^{\Katz}(\psi_{\hg h}(\ell,s))$ and $L_p^{\Katz}(\psi_{\hg h'}(\ell,s))$. At such critical pairs, it is readily verified that the following factorization of classical $L$-values occurs up to an $L$-rational fudge factor: $$L(g_\ell,h,s) = L(\psi_{gh}(\ell,s)^{-1},0)  L(\psi_{gh'}(\ell,s)^{-1},0).$$
Using this identity, the same computations as in \cite[Theorem 4.2]{DLR2} show that
there is an $L$-rational fudge factor $\mathfrak f(\ell,s) \in \Lambda^{\otimes 2}$ such that
\begin{equation}\label{fact-2-var}
L_p(\hg,h)(\ell,s) \cdot L_p^{\Katz}(\psi_g(\ell)) = \mathfrak f(\ell,s) \cdot L_p^{\Katz}(\psi_{gh}(\ell,s)) \cdot L_p^{\Katz}(\psi_{gh'}(\ell,s)).
\end{equation}

If we now restrict to $\ell=1$ and invoke Katz's $p$-adic analogue of Kronecker limit formula which asserts that $L_p^{\Katz}(\psi) =  \log_p(u_{\psi_{\ad}}) \, (\mathrm{mod} \, L^\times)$, it follows that
\begin{equation}
L_p(g,g^*,s) = \frac{ \mathfrak f(s)}{\log_p(u_{\psi_{\ad}})} \cdot L_p^{\Katz}(\mathbb N^{s}) \cdot L_p^{\Katz}(\psi_{\ad},s).
\end{equation}


Finally, Gross's main theorem in \cite{Gr} together with the functional equation for Kubota-Leopoldt's $p$-adic $L$-function asserts that
\begin{equation}\label{gross}
L_p^{\Katz}(s) = \zeta_p(s) \cdot L_p(\chi_K \omega,s) \end{equation}
up to a rational fudge factor. This yields the theorem.
\end{proof}


Assume now that $g=\theta(\psi)$ is the theta series of a character of a real quadratic field in which $p$ is inert. In light of Theorem \ref{factorizationCM-in-zetas} it is natural to pose the following question:
\begin{question}
Assume $K$ is real and $p$ remains inert in $K$. Let $u_K$ be a fundamental unit of $K$ and let $L_p(\psi_{\ad} \omega,s)$ denote the Deligne-Ribet $p$-adic $L$-function attached to $\psi_{\ad} \omega$.  Is it true that  \begin{equation}\label{factorizationRM-in-zetas}
L_p(g,g^*,s) \stackrel{?}{=} \frac{1}{\log_p(u_K)} \cdot \zeta_{p}(s) \cdot L_p(\chi_K,s) \cdot L_p(\psi_{\ad} \omega, s)
\end{equation}
up to an $L$-rational fudge factor?
\end{question}

Note that the results of this paper, combined with Darmon--Dasgupta-Pollack's \cite[Theorem 2]{DDP} prove that the above factorization holds  when evaluated at $s=0$ and $s=1$. Indeed, Theorem A in this setting takes the simple form
 \[ L_p(g,g^*,1) = \log_p(v_1)  \pmod{L^{\times}}, \] while \cite[Theorem 2]{DDP} asserts that $L_p(\psi_{\ad} \omega,s)$ vanishes at $s=0$ and \[ L_p'(\psi_{\ad} \omega,0) = \log_p(v_1) \pmod{L^{\times}}. \]
Moreover, $\zeta_p(s)$ has a simple pole at $s=0$ whose residue is a non-zero rational number, and Leopoldt's formula asserts that $L_p(\chi_K,0) = \log_p(u_K) \, (\mathrm{mod} \, L^\times)$.  Putting all together shows that  \eqref{factorizationRM-in-zetas} is true at $s=0$.  The functional equations satisfied by each of the $p$-adic $L$-functions in play ensure that the same is true at $s=1$. This of course falls short from establishing \eqref{factorizationRM-in-zetas}.


\begin{thebibliography}{RiRo1}

\bibitem[AL]{AL}
A.~Atkin, W.C.~Winnie Li, {\em Twists of newforms and pseudo-eigenvalues of $W$-operators}, Invent. Math. {\bf 48} (1978), no. 3, 221--243.

\bibitem[BDR1]{BDR2}
M.~Bertolini, H.~Darmon, V.~Rotger {\em Beilinson-Flach elements and Euler systems II: $p$-adic families and the Birch and Swinnerton-Dyer conjecture}, {\em J. Algebraic Geometry} {\bf 24} (2015), 569--604.

\bibitem[BeDi]{BeDi}
J.~Bellaiche, M.~Dimitrov, {\em On the eigencurve at classical weight one points}, {\em Duke Math. J.} {\bf 165} (2016), 245--266.

\bibitem[Bel]{Bel}
J.~Bellaiche, {\em An introduction to the conjecture of Bloch and Kato}, available at \url{http://people.brandeis.edu/~jbellaic}.

\bibitem[BDP]{BDP}
A.~Betina, M.~Dimitrov, and A.~Pozzi. {\em On the Gross-Stark conjecture}. Preprint.

\bibitem[BK]{BK}
S.~Bloch and K.~Kato. {\em $L$-functions and Tamagawa numbers of motives}, {\em The Grothendieck Festschrift I, Progr. Math.} {\bf 108} (1993), 333--400, Birkhauser.

\bibitem[BSV]{BSV}
M.~Bertolini, M.~Seveso, and R.~Venerucci. {\em Balanced diagonal classes and rational points on elliptic curves}, submitted in the collective volume "Heegner points, Stark-Heegner points and diagonal classes" together with H.\,Darmon and V.\,Rotger.

\bibitem[CaHs]{CH}
F.~Castell\`a and M.-L., Hsieh. {\em On the On the non-vanishing of generalized Kato classes for elliptic curves of rank two}. Preprint.

\bibitem[Ci]{Ci}
C.~Citro. {\em $L$-invariants of adjoint square Galois representations coming from modular forms}, {\em Int. Math Res. Not.} {\bf 14} (2008).

\bibitem[DD]{DD}
H.~Darmon and S.~Dasgupta. {\em Elliptic units for real quadratic fields}, {\em Annals Math.} {\bf 163} (2006), 301--346.

\bibitem[DDP]{DDP}
H.~Darmon, S.~Dasgupta, and R.~Pollack. {\em Hilbert modular forms and the Gross-Stark conjecture}, {\em Annals Math.} {\bf 174} (2011), 439--484.

\bibitem[DLR1]{DLR1}
H.~Darmon, A.~Lauder, and V.~Rotger.
{\em Stark points and $p$-adic iterated integrals attached to modular forms of weight one}, {\em Forum Math., Pi} (2015).

\bibitem[DLR2]{DLR2}
H.~Darmon, A.~Lauder, and V.~Rotger.
{\em Gross-Stark units and $p$-adic iterated integrals attached to modular forms of weight one}, {\em Ann. Math. Qu\'ebec}, volume dedicated to Prof. Glenn Stevens on his 60th birthday, {\bf 40} (2016), 325--354.

\bibitem[DLR3]{DLR3}
H.~Darmon, A.~Lauder, and V.~Rotger. {\em Overconvergent generalised eigenforms of weight one and class fields of real quadratic fields}, {\em Advances Math.} {\bf 283} (2015), 130--142.

\bibitem[DLR4]{DLR4}
H.~Darmon, A.~Lauder, and V.~Rotger.
{\em First order $p$-adic deformations of weight one newforms}, in ``$L$-functions and automorphic forms", {\em Contr. in Math. and Comp. Sc. {\bf 12}}.

\bibitem[DR1]{DR2}
H.~Darmon and V.~Rotger. {\em Diagonal cycles and Euler systems II: the Birch and Swinnerton-Dyer conjecture for Hasse-Weil-Artin $L$-series}, {\em Journal Amer. Math. Soc.} {\bf 30} (2017), 601--672.

\bibitem[DR2]{DR2.5}
H.~Darmon and V.~Rotger. {\em Elliptic curves of rank two and generalised Kato classes}, {\em Research in Math. Sciences} {\bf 3:27} (2016).

\bibitem[DR3]{DR3}
H.~Darmon and V.~Rotger, {\em Stark-Heegner points and diagonal cycles}, submitted in the collective volume ``Heegner points, Stark-Heegner points and diagonal classes" together with M.\,Bertolini, M.\,A.\,Seveso and R.\,Venerucci.

\bibitem[DHRV]{DHRV}
H.~Darmon, M.~Harris, V.~Rotger, and A.~Venkatesh, {\em Derived Hecke algebra for dihedral weight one forms}, in progress.

\bibitem[Das1]{Das}
S.~Dasgupta, {\em Stark's Conjectures}. Senior honors thesis, Harvard University (1999).

\bibitem[Das2]{Das3}
S.~Dasgupta, {\em Factorization of $p$-adic Rankin $L$-series}, {\em Invent. Math.} {\bf 205} (2016), no. 1, 221--268.

\bibitem[Gr]{Gr}
B.~Gross, {\em On the factorization of $p$-adic $L$-series.} {\em Invent. Math.} {\bf 57} (1980), no. 1, 83--95.

\bibitem[GS]{GS}
R.~Greenberg and G.~Stevens. {\em On the Conjecture of Mazur, Tate, and Teitelbaum}, {\em Contemporary Mathematics} {\bf 165} (1994).

\bibitem[Hi1]{Hi2}
H.~Hida, {\em A $p$-adic measure attached to the zeta functions associated with two elliptic modular forms I}, {\em Invent. Math.} {\bf 79} (1985), no. 1, 159--195.

\bibitem[Hi2]{Hi3}
H.~Hida, {\em A $p$-adic measure attached to the zeta functions associated with two elliptic modular forms II}, {\em Ann. Inst. Fourier (Grenoble)} {\bf 38} (1988), no. 3, 1--83.

\bibitem[Hi3]{Hi-ad}
H.~Hida. {\em Greenberg's $\cL$-invariants of adjoint square Galois representations}, {\em IMRN} {\bf 59} (2004), 3177--3189.

\bibitem[KLZ2]{KLZ}
G.~Kings, D.~Loeffler, and S.L.~Zerbes, {\em Rankin-Eisenstein classes and explicit reciprocity laws}, {\em Cambridge J. Math} {\bf 5} (2017), no. 1, 1--122.

\bibitem[KLZ1]{KLZ1}
G.~Kings, D.~Loeffler, and S.L.~Zerbes, {\em Rankin-Eisenstein classes for modular forms}, {\em American J. Math.} {\bf 142} (2020), no. 1, 79--138.

\bibitem[LLZ]{LLZ}
A.~Lei, D.~Loeffler, and S.L.~Zerbes, {\em Euler systems for Rankin Selberg convolutions}, {\em Annals Math.} {\bf 180} (2014), no. 2, 653--771.

\bibitem[LZ1]{LZ2}
D.~Loeffler, and S.L.~Zerbes, {\em Iwasawa theory fo the symmetric square of a modular form}, {\em J. Reine Angew. Math.} {\bf 752} (2019), 179--210.

\bibitem[MTT]{MTT}
B.~Mazur, J.~Tate, and J.~Teitelbaum, {\em On $p$-adic analogs of the conjectures of Birch and Swinertonn-Dyer}, {\em Invent. Math.} {\bf 84} (1986), 1--48.

\bibitem[Nek]{Nek}
J.~Nekov\'ar, {\em $p$-adic Abel-Jacobi maps and $p$-adic heights}. Lecture notes of the author's lecture at the 1998 Conference ``The Arithmetic and Geometry of Algebraic Cycles".

\bibitem[Och]{Och}
T.~Ochiai, {\em A generalization of the Coleman map for Hida deformations}, {\em Amer. J. Math.} {\bf 125} (2003), no. 4.

\bibitem[Oh]{Oh}
M.~Ohta, {\em Ordinary $p$-adic \'etale cohomology groups attached to towers of ellimptic modular curves II}, {\em Math. Annalen} {\bf 318} (2000), no. 3, 557--583.

\bibitem[Park]{Park}
J.~Park, {\em The Darmon--Dasgupta units over genus fields and the Shimura correspondence,} {\em J. Number Theory} {\bf 130} (2010), no. 11, 2610--2627.

\bibitem[RiRo]{RR}
O.~Rivero and V.~Rotger, {\em Beilinson--Flach elements, Stark units, and $p$-adic iterated integrals}, {\em Forum Math.} {\bf 31} (2019), no. 6, 1517--1532.

\bibitem[Sc]{Sc}
C.~Schmidt, {\em $p$-adic measures attached to automorphic representations of $\GL(3)$}, {\em Invent. Math.} {\bf 92} (1988), 597--631.

\bibitem[Wi]{Wi}
A.~Wiles, {\em On ordinary $\Lambda$-adic representations assoc. to modular forms}, {\em Invent. Math.} {\bf 94} (1988), 529--573.

\end{thebibliography}
\end{document}